\documentclass[dvipsnames]{amsart}

\usepackage[T1]{fontenc}
\usepackage[utf8]{inputenc}

\usepackage{amsmath}
\usepackage{amssymb}
\usepackage{amsthm}
\usepackage{physics}  % for \eval
\usepackage{upgreek}
\usepackage{mathtools}    
\usepackage[final]{microtype}

\usepackage{tikz}
\usepackage{tikz-cd}
\usetikzlibrary{arrows, babel, 
  decorations.text, patterns, 
  patterns.meta, decorations.pathreplacing}
\tikzset{commutative diagrams/.cd,
  arrow style=tikz, diagrams={>=stealth}}

\usepackage{url}
\usepackage[%pdftex,
            pdfnewwindow=false,
            colorlinks=true,
            allcolors=black,
            urlcolor=red,
            citecolor=green!50!black]{hyperref}
\usepackage[safeinputenc, 
            backend=biber, 
            url=false,
            doi=false,
            isbn=false,
            bibencoding=auto,
            style=alphabetic, 
            giveninits=true,
            backref=true]{biblatex}
\DeclareFieldFormat[article, incollection]{title}{\mkbibemph{#1}} 
\AtEveryBibitem{%
   \ifentrytype{misc}{%
      \clearfield{year}%
   }{}%
}

\DefineBibliographyStrings{english}{%
  backrefpage = {$\uparrow$},
  backrefpages = {$\uparrow$},}

\AtBeginBibliography{\small}
\renewbibmacro{in:}{}

\usepackage{CJKutf8}
\begin{CJK*}{UTF8}{goth}
\gdef\yama{\mbox{\textbf{山}}}
\gdef\ten{\mbox{\textbf{天}}}
\end{CJK*}

\usepackage{longtable}
\usepackage{booktabs}
\usepackage{caption}
\usepackage{multirow}

\usepackage{enumitem}
\setenumerate[1]{label={{\bf \arabic*}.}, itemsep=3pt}
\setenumerate[2]{label=(\alph*)}
\numberwithin{equation}{section}
\numberwithin{table}{section}

\usepackage{fancyhdr}

\usepackage[includehead, includefoot, 
  left=3cm, right=3cm, top=2cm, bottom=3cm, 
  footskip=5mm]{geometry}

\theoremstyle{plain}
\newtheorem{theorem}{Theorem}[section]
  \newtheorem{proposition}[theorem]{Proposition}
  \newtheorem{lemma}[theorem]{Lemma}
  \newtheorem{corollary}[theorem]{Corollary}
  \newtheorem{conjecture}[theorem]{Conjecture}

\newtheorem*{theorem*}{Theorem}
\newtheorem{introtheorem}{Theorem}

  \newtheorem{introquestion}[introtheorem]{Question}

\theoremstyle{remark}
\newtheorem{remark}[theorem]{Remark}
\newtheorem*{remark*}{Remark}
\newcommand{\FF}{\mathbb{F}}
\newcommand{\NN}{\mathbb{N}}
\newcommand{\PP}{\mathbb{P}}
\newcommand{\QQ}{\mathbb{Q}}

\newcommand{\ZZ}{\mathbb{Z}}

\newcommand{\Qbar}{\overline{\mathbb{Q}}}

\newcommand{\OO}{\mathcal{O}}
\newcommand{\OK}{\mathcal{O}_K}
\newcommand{\pp}{\mathfrak{p}}
\newcommand{\bmu}{\mathbf{\upmu}}
\newcommand{\mf}{\mathfrak}

\DeclareMathOperator{\Aut}{\textnormal{\textsf{Aut}}}

\DeclareMathOperator{\Gal}{\textnormal{\textsf{Gal}}}
\DeclareMathOperator{\GL}{\textnormal{\textsf{GL}}}

\DeclareMathOperator{\ord}{\textnormal{\textsf{ord}}}

\let\trace\relax
\DeclareMathOperator{\trace}{\textnormal{\textsf{tr}}}

\newcommand{\Sage}{{\tt SageMath\ }}
\newcommand{\Mod}[1]{\ (\mathrm{mod}\ #1)}
\newcommand{\defeq}{\vcentcolon=}

\title{Heavenly Elliptic Curves over Cubic Number Fields}

\author[S.~O'Hara]{Suzanne O'Hara}
\address{Department of Mathematics and Computer Science \\ Wesleyan University \\ Middletown CT, 06459 \\ United States}
\email{seohara@wesleyan.edu}

\begin{document}

\begin{abstract}
  The study of heavenly abelian varieties is motivated by a question of Ihara. When an elliptic curve $E/K$ is heavenly at $\ell$, the extension $K(E[\ell^\infty])/K(\bmu_\ell^\infty)$ is pro-$\ell$ and unramified away from $\ell$. These are the same arithmetic conditions as the fixed field of the kernel attached to pro-$\ell$ \'etale covers of the projective line over $K$ minus three points. In this paper we study heavenly elliptic curves defined over cubic number fields. Following the work on McLeman and Rasmussen in the quadratic case, we find that there is a wider range of possible behaviors for the trace of a \textit{balanced} elliptic curve in the cubic case. In this setting, we introduce a further distinction between balanced and \textit{totally balanced} curves. With this, we show that trace of the representation on the $\ell$-torsion for totally balanced elliptic curves is non-surjective. We also compute a bound on primes $\ell$ after which any heavenly elliptic curve defined over a cubic number field must be balanced. Finally we compare the trace behavior of totally balanced elliptic curves with CM elliptic curves.

\end{abstract}

\begin{CJK*}{UTF8}{goth}

\maketitle

\section{Introduction}\label{sec:intro}

    Let us begin by fixing an algebraic closure $\Qbar$ of $\QQ$, so that any algebraic extension of $K/\QQ$ may be viewed as sitting inside of $\Qbar$. For any integer $N > 1$ let $\bmu_N$ denote the $N$-th roots of unity, and $\bmu_{N^\infty} = \bigcup_{i} \bmu_{N^i}$. For any number field $K$ and rational prime $\ell$, we can define the following fields: $\ten \defeq \ten(K,\ell)$\footnote{The kanji \ten \, is pronounced "ten" and translates to \textit{heaven}.} the maximal pro-$\ell$ extension of $K(\bmu_\ell^\infty)$ unramified away from $\ell$, and $\yama \defeq \yama(K,\ell)$\footnote{The kanji \yama \, is pronounced "yama" and tranlates to \textit{mountain}.} the fixed field of the kernel of the outer Galois representation attached to pro-$\ell$ covers of $\PP_K^1-\{0,1,\infty\}$. We call an abelian variety $A/K$ \textit{heavenly} at $\ell$ if $K(A[\ell^\infty]) \subseteq \ten$. As being heavenly is an isomorphism invariant over the field of definition, we consider varieties up to $K$-isomorphism class, denoted $[A]_K$.

For any number field $K$ and dimension $g > 0$ we let $\mathcal{H}(K,g)$ be the set of pairs $([A]_K,\ell)$ such that $A/K$ is heavenly at $\ell$. When $K$ is a cubic number field, it is known that $\mathcal{H}(K,1)$ is finite by \cite[Prop. 7.5]{RT-constrainedtorsion}. Based on the work of McLeman and Rasmussen in the quadratic \cite{quadfields}, we now wish to consider whether there remain finitely many examples of heavenly elliptic curves for any prime $\ell$ as we let $K$ vary through all Galois cubic fields. The reduction to Galois fields is needed to show that \textit{totally balanced} curves have similar trace behavior to elliptic curves with complex multiplication. We show that when a heavenly elliptic curve is totally balanced, the representation of the Galois group $G_K$ on the $\ell$-torsion of the curve has a non-surjective trace modulo $\ell$.

We define the following family of sets. Let $\mathcal{H}^\circ(n,1)_B$ be the set of pairs $([E]_K,\ell)$ where
    \begin{enumerate}
        \item[(a)] $\ell \geq B$ is prime,
        \item[(b)] $[K:\QQ] = n$ is a Galois extension,
        \item[(c)] $E/K$ is an elliptic curve heavenly at $\ell$,
        \item[(d)] $E\times_K\Qbar \not\cong E_0 \times_\QQ \Qbar $ for any $E_0/\QQ$ which is heavenly at $\ell$.
    \end{enumerate}

A main question of interest is the following:

\begin{introquestion}\label{introq: H finite}
Is $\mathcal{H}^\circ (3,1)_{11}$ a finite set?
\end{introquestion}

In section \ref{sec:finite-fibers} we will do a further exploration of $\mathcal{H}^\circ(n,1)_B$ for further pairs of integers $B,n$ where $n$ is odd. We see this restriction to primes $\ell \geq 11$, as \cite[Thm. 3.2]{quadfields} shows that it is possible to parameterize infinite families of heavenly curves for $\ell \leq 7$. We also remove curves that arise as the base-change of some example over $\QQ$ to avoid infinite families of curves. If $E/\QQ$ is heavenly at $\ell$ and $K/\QQ$ is a finite degree extension, then $E \times_\QQ K$ is heavenly at $\ell$ as well. The set $\mathcal{H}(\QQ,1)$ is known explicitly. So it is possible to pick some $([E']_\QQ, \ell) \in \mathcal{H}(\QQ,1)$ and an infinite set of distinct cubic fields $\{K_i\}_{i \in \mathbb{N}}$ and generate infinitely many distinct pairs of the form $([E'\times_\QQ K_i]_{K_i}, \ell)$ which are heavenly at $\ell$ over a cubic field.

Instead of indexing the curves themselves, we may also choose to index over the possible fields where examples of heavenly elliptic curves my arise. Let $\mathcal{R}(n,1)_B$ be the set of pairs $(\ell,K)$ such that
    \begin{enumerate}
        \item[(a)] $\ell \geq B$ is prime,
        \item[(b)] $[K:\QQ] = n$ is a Galois extension,
        \item[(c)] There exists an elliptic curves $E/K$ heavenly at $\ell$, such that
        \item[(d)] $E\times_K\Qbar \not\cong E_0 \times_\QQ \Qbar $ for any $E_0/\QQ$ which is heavenly at $\ell$.
    \end{enumerate}

\begin{introquestion}\label{introq: R finite}

    Is $\mathcal{R}(3,1)_{11}$ a finite set?
\end{introquestion}

Note that a pair $(\ell, K) \in \mathcal{R}(n,1)_B$ exactly when there exists a pair $([A]_K,\ell) \in \mathcal{H}(n,1)_B$. 

A positive answer to either questions A or B are equivalent to each other. This is easy to see as if $\sharp\mathcal{H}^\circ(3,1)_{11} < \infty$, then there are necessarily only finitely many pairs $(\ell, K) \in \mathcal{R}(3,1)_{11}$ by the observation above. If $\sharp\mathcal{R}(3,1)_{11} < \infty$, then we can argue each pair $(\ell, K)$ gives rise to only finitely many pairs $([A]_K,\ell) \in \mathcal{H}^\circ(3,1)_{11}$. This is because having good reduction away from $\ell$ is a necessary condition for a curve $A/K$ to be heavenly. The Shafarevich Conjecture (\cite{faltings83}, \cite{zahrin:unpolarized}) guarantees there only exists finitely many $K$-isomorphism classes of elliptic curves with good reduction away from $\ell$.

The figure below illustrates a region representing the pairs $(\ell,K) \in \mathcal{R}(3,1)_{11}$ as a subset of $\NN \times \mathcal{F}$ where $\mathcal{F}$ orders Galois cubic fields by discriminant, breaking ties arbitrarily. As noted before, there may be infinite families of heavenly elliptic curves for each $\ell$ when $\ell < 11$ so this region is indicated in red. The existence of bounding curve to the right is a visualization of the previously noted result \cite[Prop. 7.5]{RT-constrainedtorsion}. We have $\sharp \mathcal{H}(K,1) < \infty$ for any cubic field. Hence there may be some biggest prime $\ell(K)$ past which there are no examples of heavenly curves $E/K$ with $\ell > \ell(K)$. Since there are only finitely many cubic fields of some fixed discriminant $\Delta$,\cite[Thm 2.16]{algebraic_number_neukrich} it is simple to note then that we may extend this work a bound $\ell(\Delta)$. Therefore, for each discriminant $\Delta$, there are no examples of elliptic curves defined over cubic fields which are heavenly at $\ell > \ell(\Delta)$. This is indicated by the shaded gray region to the right of this curve.

We may then address the vertical bounding curve labeled \cite[Thm. 3.1]{quadfields}. This result implies that over all elliptic curves defined over a cubic number field, there are only finitely many, up to $\Qbar$ isomorphism classes, which are heavenly at some fixed $\ell \geq 11$. Now we may say there is also a bound on discriminants,  $\Delta(\ell)$, which depends on our chosen prime $\ell$ such that for any $\Delta > \Delta(\ell)$ there are no examples of elliptic curves $E/K$ which are heavenly at $\ell$ with the cubic field $K$ having discriminant $\Delta$. We again indicate this with the shaded gray region above the curve.

These two results can be though of as "horizontal" and "vertical" finiteness results respectively. Restricting to considering primes $\ell \geq 11$, visualize this as any vertical or horizontal line passing through this figure intersecting only finitely many ordered pairs $(\ell, K)$. Hence showing that the region $\mathcal{R}(3,1)_{11}$ is also bounded would be enough to conclude that there are only finitely many classes of heavenly elliptic curves over cubic Galois fields up to $\Qbar$ isomorphism. Extending this result to counting $K$-isomorphism classes is done in \S \ref{sec:finite-fibers}.

In \S \ref{sec:arithmetic} we prove a uniform bound after which all heavenly elliptic curves must be balanced. See \cite{OHara26} for the \texttt{SageMath} code used in these calculations. This is illustrated in Figure \ref{fig:H31_map} by the dotted line. This bound is uniform as it only depends upon the dimension of the variety and the degree of the field. The other vertical line illustrates a result of Bourdon.  For any $\ell > 907$, there does not exist a CM elliptic curve defined over a cubic field $K$ with $K(E[\ell^\infty])$ a pro-$\ell$ extension of $K(\mu_\ell)$ \cite{bourdon:uniformCM}. An immediate consequence of Bourdon's result is that there are no heavenly CM curves with $\ell > 907$.

At the current moment, it is not known if we can compute a similar bound to determine when curves must be totally balanced. Such a result would be needed if we wish to use Bourdon's result to complete a bounded region for $\mathcal{R}(3,1)_{11}$, as we do not expect only knowing a curve is balanced is enough to ensure it has a trace that behaves similarly to CM curves.

%Existence of Bourdon's bound is Theorem 1, but I haven't listed a specific result here because the computatino for n = 3 is not done until later in the paper.

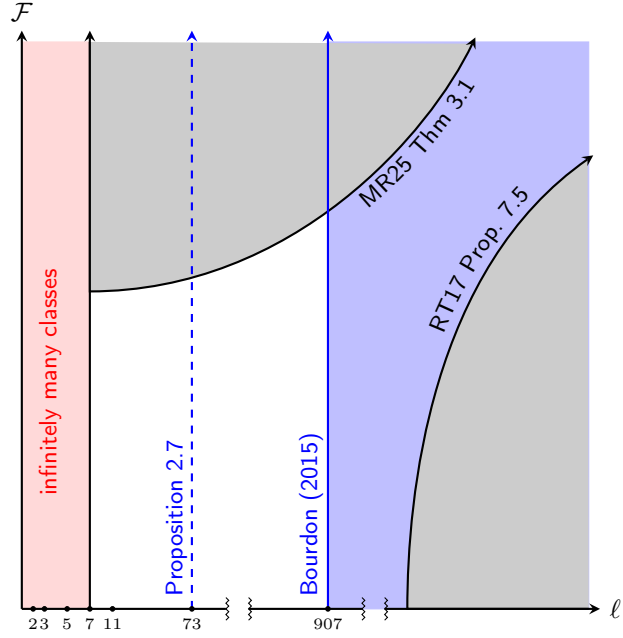
\begin{figure}[ht!]
\centering
\begin{tikzpicture}[>=stealth, scale=0.15]

% tiny ell strip
\draw[fill=red!15, red!15] (0,0) rectangle (6, 50);
\node[fill=red!15] at (2.5, 20) {\rotatebox{90}{\small \textcolor{red}{\sffamily infinitely many classes}}};

% tiny ell line
\draw[->, thick, black] (6,0) -- (6,51);
  
% balanced line
%\draw[fill=blue!10, blue!10] (5,20) rectangle (50,50);
%\draw[fill=blue!10, blue!10] (27, 0) rectangle (50, 50);
\node[blue, anchor=south east] at (15.4, 0.3) {\rotatebox{90}{\small \sffamily Proposition 2.7}};

% Bourdon-CM line -- fill blue first, then add vertical line, text
\draw[fill=blue!15, blue!25] (27,0) rectangle (50,50);
\node[blue, anchor=south east] at (27.4, 0.3) {\rotatebox{90}{\small \sffamily Bourdon (2015)}};
%  \node[blue, anchor=south,text width=.5in] at (27.4, 1.4) {\rotatebox{90}{\tiny complex}};

% Bach-Sorenson
  \draw[->, thick, black, fill=gray!40,
  , postaction={decorate, decoration={raise=1.0ex, text along path, text align={left indent={0.65\dimexpr\pgfdecoratedpathlength\relax}}, text={|\small\sffamily| RT17 Prop. 7.5}}}] 
  (50,0) -- (34, 0) .. controls (34, 10) and (36, 30) .. (50.4, 40);
%  \node at (44, 15) {\large $X_{1}$};

  % McLeman-Rasmussen
  \draw[->, thick, black, fill=gray!40, postaction={decorate, decoration={raise=-2.0ex, text along path, text align={left indent={0.73\dimexpr\pgfdecoratedpathlength\relax}}, text={|\small\sffamily|MR25 Thm 3.1}}} ] (6, 50) -- (6, 28) .. controls (24, 28) and (36, 42) .. (40, 50.4);
%  \draw[->, thick, red] (5, 0) -- (5, 51);

%  \node at (20, 40) {\large $X_{2}$};
%  \node at (32, 28) {\large $X_{3}$};
  
%%%  Some vertical lines

  % balanced / Conjecture A
  \draw[->, thick, dashed, blue] (15,0) -- (15,51);

  % CM / Bourdon
  % \draw[thick, blue] (27,0) -- (27,32.85);
  % \draw[->, thick, blue] (27, 35.1) -- (27, 51);
  \draw[->, thick, blue] (27, 0) -- (27, 51);

  % Daniels-Lozano-Robledo
%  \draw[->, thick, purple] (5,20) -- (51,20);
  
  % x-axis (the prime ell)
  \draw[thick, black] (0,0) -- (18,0);
  \draw[thick, black] (20,0) -- (30,0);
  \draw[->, thick, black] (32,0) -- (51,0) node[right] {$\ell$};

  \foreach \xcoord/\xlabel in {1/2, 2/3, 4/5, 6/7, 8/11, 15/73, 27/907} {
    \draw[black, fill] (\xcoord, 0) circle (0.2) node[below] {\tiny $\xlabel$}; % node[above, gray] {\tiny $\xcoord$};
  }

%  \draw[black, fill] (0, 20) circle (0.2) node[left] {\tiny $89$};

  % breaks
  \draw[thin, black] (18, 1) --
      ++(0.25, -0.25) -- ++(-0.25, -0.25) --
      ++(0.25, -0.25) -- ++(-0.25, -0.25) --
      ++(0.25, -0.25) -- ++(-0.25, -0.25) --
      ++(0.25, -0.25) -- ++(-0.25, -0.25);

  \draw[thin, black] (20, 1) --
      ++(0.25, -0.25) -- ++(-0.25, -0.25) --
      ++(0.25, -0.25) -- ++(-0.25, -0.25) --
      ++(0.25, -0.25) -- ++(-0.25, -0.25) --
      ++(0.25, -0.25) -- ++(-0.25, -0.25);

  \draw[thin, black] (30, 1) --
      ++(0.25, -0.25) -- ++(-0.25, -0.25) --
      ++(0.25, -0.25) -- ++(-0.25, -0.25) --
      ++(0.25, -0.25) -- ++(-0.25, -0.25) --
      ++(0.25, -0.25) -- ++(-0.25, -0.25);

  \draw[thin, black] (32, 1) --
      ++(0.25, -0.25) -- ++(-0.25, -0.25) --
      ++(0.25, -0.25) -- ++(-0.25, -0.25) --
      ++(0.25, -0.25) -- ++(-0.25, -0.25) --
      ++(0.25, -0.25) -- ++(-0.25, -0.25);
  
  % y-axis (discriminant)
  \draw[->, thick, black] (0, 0) -- (0, 51) node[above] {$\mathcal{F}$};
  % \node[black, anchor=south east] at (0, 12) { \rotatebox{90}{{\tiny absolute value quadratic fields $k$ (ordered by increasing values of $\left|\Delta_{k/\QQ}\right|$)}} };
\end{tikzpicture}
\caption{\small The set $\mathcal{R}(3,1)_{11}$ inside $\NN \times \mathcal{F}$}\label{fig:H31_map}
%is contained within the blue and white regions.}\label{fig:H21_map}
\end{figure}

It remains an open question whether for a fixed prime $\ell$ and $K$ a Galois extension of $\QQ$ of odd degree, if all elliptic curves $E/K$ which are totally balanced at $\ell$ must also have complex multiplication. This study differs from the quadratic case in this aspect. When working over a quadratic extension of $\QQ$, knowing that an elliptic curve is balanced is enough to determine uniquely the form of the Galois representation on the $\ell$-torsion points. However, it may not be the case knowing a curve $E$ is balanced lets us conclude that a unique form of the representation on $\ell$-torsion exists when our number field is of higher degree. This distinction is discussed more thoroughly in \S \ref{sec:odd-degree}. 

\section*{Acknowledgments}

I am incredibly thankful to my Ph.D. advisor Christopher Rasmussen for all of his support and guidance during this project. I would also like to thank Cam McLeman for his assistance in understanding arguments using Gr{\"o}\ss encharacters.

\section*{AI Statement}
No AI was used in the writing or editing of this paper or the corresponding code.
    
\section{Arithmetic of Heavenly Varieties}\label{sec:arithmetic}

    \subsection{General Notation.}

For a number field $K$, let $G_K := \Gal(\Qbar/K)$ and let $\OK$ denote the ring of integers of $K$. For a prime number $\ell$, let $\chi: G_\QQ \to \FF_\ell^\times$ be the $\ell$-adic cyclotomic character modulo $\ell$. As is well known, for all $\zeta \in \mu_\ell$ and $\sigma \in G_\QQ$ we have $\zeta^\sigma = \zeta^{\chi(\sigma)}$. Let $L/K$ be a finite extension of fields and suppose $\mf{P}$ is a prime of $L$ above $\pp$. We denote the ramification index of $\mf{P}$ over $\pp$ as $e_{\mf{P}\vert\pp}$ and the inertial degree, $[\OO_L/\mf{P} : \OK/\pp]$, as $f_{\mf{P}\vert\pp}$. However in the case where the base field $K = \QQ$, we simply write $e_\mf{P}$ and $f_\mf{P}$. For a prime $\pp$ of $K$, we let $\textbf{N}\pp$ be the absolute norm of $\pp$ and $\FF_\pp$ denotes the finite field $\mathcal{O}_K/\pp$. Finally, let $K_\mf{p}$ denote the completion of $K$ at the prime ideal $\mf{p}$.

For an abelian variety $A/K$ and integer $m \geq 1$. We denote by $A[m]$ the group of $m$-torsion points of $A$, and we let $\rho_{A,m} \colon G_K \to \mathrm{GL}_{2g}(\ZZ/m\ZZ)$ be the $G_K$ representation of the $m$-torsion. Such a representation requires a choice of basis for $\mathrm{GL}_{2g}(\ZZ/m\ZZ)$, hence the image of such a map is only defined up to conjugacy. 

For an elliptic curve $E/K$ with conductor $\mf{N}$ and good reduction at $\pp$, $a_\pp = a_\pp(E) := \textbf{N}_\pp + 1 - \sharp E(\FF_\pp)$. For $\theta_\pp \in G_K$, a Frobenius element of $\pp$, it is well known that $a_\pp \equiv \rho_{E,\ell}(\theta_\pp) \Mod{\ell}$. The map $\rho_{E,\ell}$ maps into a finite group, and hence the representation must factor through a finite index subgroup of $G_K$, namely $\ker(\rho_{E,\ell})$. The Chebortarev Density Theorem \cite[Thm. 13.4]{algebraic_number_neukrich} guarantees that each coset of $\ker(\rho_{E,\ell})$ contains a representative that is a Frobenius for some prime. In fact, as E may have bad reduction at a finite number of places, we can guarantee that each coset has a Frobenius representative at a prime of good reduction. As $\ker(\rho_{E,\ell})$ has only finitely many cosets, and the Frobenius elements are dense in $G_K$, each coset $\sigma \ker(\rho_{E,\ell})$ must contain infinitely many Frobenius representatives. This still holds after we exclude the finitely many Frobenius elements for the primes of bad reduction.  Consequently, $\trace(\rho_{E,\ell}(G_K)) = \{a_\pp(E) \Mod{\ell} : \pp \nmid  \ell\mf{N}\}$.

For an abelian variety $A/K$, let $e(A/K;\mf{l})$ denote the index of semistable reduction for $A$ over the localization of $K$ at $\mf{l}$. That is, for $K_\mf{l}^\mathrm{unr}$ the maximal unramified extension of $K_\mf{l}$, let $K_\mf{l}^\mathrm{ss}$ be the minimal extension of $K_\mf{l}^\mathrm{unr}$ over which $A$ obtains semistable reduction. Then $e(A/K;\mf{l})$ is the degree $[K_\mf{l}^\mathrm{ss}:K_\mf{l}^\mathrm{unr}]$\cite{SGA7-I}. We set $e_A(\mf{l}) : = e_{\mf{l}}\cdot e(A/K;\mf{l})$. When there is no confusion on the prime and abelian variety of interest, we may simply write $e$ for $e_A(\mf{l})$.

We will record here for future reference a result that enumerates the possible index of semistable reduction for elliptic curves.

\begin{proposition}\cite[Expos\'e IX]{SGA7-I}\label{prop: semistable options}
    Let $E/K$ be an elliptic curve, $\ell > 3$ a prime, and $\mf{l}$ a prime of $K$ above $\ell$. Then $e(E/K;\mf{l}) \in \{1,2,3,4,6\}$.
\end{proposition}

\subsection{Arithmetic of Heavenly Varieties}
While the original definition of an elliptic curve $E/K$ being heavenly at $\ell$ is motivated by the construction of $\ten$, there is an equivalent definition which we utilize throughout the paper. While this paper focuses only on elliptic curves, we state the results in this section for abelian varieties of any dimension $g$. The lemma below can be traced back to \cite{RT-constrainedtorsion}, however we point to \cite{quadfields} for a self-contained proof.
\begin{lemma}\cite[Lem. 2.4]{quadfields}\label{lem:heavenly equiv def}
    Let $K$ be a number field and $A/K$ an abelian variety. Let $\ell$ be a rational prime and $S$ the set of prime of $K$ dividing $\ell$. Then $A$ is heavenly at $\ell$ if and only if
    \begin{enumerate}
        \item[(a)] $A$ has good reduction outside of $S$, and
        \item[(b)] $[K(A[\ell]):K(\mu_\ell)]$ is a power of $\ell$.
    \end{enumerate}
\end{lemma}

For any number field $K$, $g > 0$, and prime $\ell$, we denote by $\mathcal{H}(K,g,\ell)$ the set of $K$-isomorphism classes of abelian varieties of dimension $g$ which are heavenly at $\ell$. The work of Faltings and Zarhrin on the Shafarevich Conjucture (\cite{faltings83}, \cite{zahrin:unpolarized}) guarantees that for a fixed field $K$, dimension $g$, and finite set of primes $S$ the number of $K$-isomorphism classes of abelian varieties of dimension $g$ with good reduction away from $S$ must be finite. Therefore Lemma \ref{lem:heavenly equiv def} implies that the set $\mathcal{H}(K,g,\ell)$ must always be finite. Rasmussen and Tamagawa conjecture that a finiteness result holds even when the prime $\ell$ is allowed to vary. Let $\mathcal{H}(K,g) \defeq \{([A]_K,\ell) : A \in \mathcal{H}(K,g,\ell)\}$. It is important to index both the curve and the prime as there are examples of curves,necessarily with everywhere good reduction, which are heavenly at more than one prime. In principal, this can occur in the cubic case as it does in the quadratic case. The interested reader can find an example defined over $\QQ(\sqrt{6})$ in \cite[\S 6]{quadfields}.

\begin{conjecture}\cite{RT-constrainedtorsion}
    Let $K$ be a number field and $g > 0$. The set $\mathcal{H}(K,g)$ is finite. Equivalently, $\mathcal{H}(K,g,\ell) = \varnothing$ for $\ell \gg 0$.
\end{conjecture}

Under the assumption of the generalized Riemann hypothesis, \cite{RT-constrainedtorsion} proves this conjecture. While this conjecture has also been proven unconditionally for a variety of cases, it is of particular interest to this paper that $\sharp\mathcal{H}(K,1) < \infty $ for $[K:\QQ] \leq 3$ without using GRH \cite{primepowertorsion, RT-constrainedtorsion}.

\begin{proposition}\cite[Lemma 3]{primepowertorsion}\label{prop: heavenly upper triangular}
    Suppose $([A]_K, \ell) \in \mathcal{H}(K,g)$. For each $r$ with $1 \leq r \leq 2g$, there exists $i_r \in \ZZ$ satisfying $0 \leq i_r < \sharp\chi(G_K)$ such that
    \[\rho_{A,\ell} \sim \begin{pmatrix}
        \chi^{i_1} & \star & \ldots & \star\\
        & \chi^{i_2} & \ldots & \star\\
        & & \ddots & \vdots\\
        & & & \chi^{i_{2g}}
    \end{pmatrix}.\]
\end{proposition}

\begin{remark}
    In the dimension 1 case we may use the Weil pairing to deduce that the determinant of $\rho_{A,\ell}$ must be the $\ell$-adic cyclotomic character, $\chi$ \cite[\S III.8]{Silverman}. We may then deduce an additional arithmetic constraint that:
    \begin{equation}\label{eq: sum i-val}
        i_1 +i_2 \equiv 1 \Mod{\ell -1}.
    \end{equation}
\end{remark}

\begin{remark}\label{rmk: pick maxl CM}
    Given two abelian varieties $A/K$ and $A'/K$ which are isogenous over $K$, then there is an equality of the torsion fields $K(A[\ell^\infty]) = K(A'[\ell^\infty])$. This, combined with the fact that two isogenous curves defined over the same field must have identical reduction information, implies that the property of being ``heavenly at $\ell$'' must be an isogeny invariant.

    Additionally, for any abelian variety $A/K$ with complex multiplication by an order in the field $L$ there must exist some $A'/K$ isogenous to $A$ with complex multiplication by the maximal order $\OO_L \subset L$ \cite[Lem. 2.4]{Lom-HeavenlyCM}. Hence when deciding whether $A/K$ with complex multiplication is heavenly, we may always choose to work with the representative $A'/K$ instead.
\end{remark}

In this paper we are concerned with bi-directional finiteness results where we will count examples of heavenly elliptic curves where both the heavenly prime $\ell$ and the field of definition $K$ are allowed to vary, given only constraints on the degree of the field $K$. We call these results bi-directional in reference to Figure \ref{fig:H31_map}, where vertical finiteness are results that bound $\Qbar$ classes of varieties at a fixed prime $\ell$, while horizontal finiteness results bound the primes $\ell$ at which heavenly varieties may exist over a fixed field.

\subsection{Balanced Abelian Varieties}

For any heavenly abelian variety $A/K$, and any prime $\mf{l}$ of $K$ above $\ell$, there exists a set of indices $j_{\mf{l},r}$ for $1 \leq r \leq 2g$ with the following property from \cite{RT-constrainedtorsion}:

\begin{equation}\label{eq: ei =j}
    ei_r \equiv j_{\mf{l},r} \Mod{\ell -1}
\end{equation}

The set $\{j_{\mf{l},r}\}_{r=1}^{2g}$ are called the \textit{Tate-Oort numbers} for $A$ which are determined by the structure of the $\ell$-torsion as a group scheme. The Tate-Oort numbers have the property that $\sum_{r = 1}^{2g}j_{\mf{l},r} = eg$. We call $A$ \textit{balanced} at $\mf{l}$ if $j_{\mf{l},r} = \frac{e}{2}$ for all $1 \leq r \leq 2g$. If $A$ is balanced at every prime $\mf{l}$ above $\ell$, then we may simply say $A$ is \textit{balanced} at $\ell$. A more thorough explanation of the theory behind this definition can be found in \cite{quadfields}. We will utilize the extra constraints on balanced curves to show such curves have a non-surjective trace in \S \ref{sec:odd-degree}. Before we begin that analysis, we give a uniform bound on $\ell$ that forces a heavenly elliptic curve over a cubic field to be balanced.

Let $\mathcal{B}(K,g)$ denote the set of pairs $([A]_K,\ell) \in \mathcal{H}(K,g)$ where $A$ is balanced at $\ell$. There exists a uniform bound, $B(n,g)$ depending only on the degree of $K/\QQ$ and the dimension $g$ such that $\ell > B(K,g)$ and $([A]_K,\ell) \in \mathcal{H}(K,g) \text{ implies } ([A]_K,\ell) \in \mathcal{B}(K,g)$ \cite[Cor. 5.3]{quadfields}.

\begin{proposition}\label{prop: B(3,1)}
    $B(3,1) \leq 73$.
\end{proposition}

\begin{proof}
    This upper bound was found by a computational search using \Sage \cite{sagemath} following a strategy described in \cite{quadfields}. The code used is publicly available at \cite{OHara26}.

In order to conduct this search, we first suppose we have a fixed pair of Tate-Oort numbers $(j_1,j_2)$ with $j_1 < j_2$. We then determine a list of primes $\ell$ for which it would be possible to have an elliptic curve, $E/K$ over a cubic field which is heavenly at $\ell$ with Tate-Oort numbers $(j_1,j_2)$. Note that the $e$ is now determined, as $j_1 + j_2 = e$.

Without loss of generality, we may assume $\ell > 11$. Fix some prime $p \neq \ell$, and for each prime of $K$, $\mathfrak{p}$ above $p$, set $f = f_\mf{p}$. Note then that $1 \leq f \leq 3$. We consider a choice of Frobenious element $\theta_\mathfrak{p}$. Let $q = p^f$. For a heavenly elliptic curve at $\ell$, the trace of our representation $\tau_1 = \alpha_1 + \alpha_2$ must have the following restraint from the Weil bound: $\vert \tau_1\vert \leq 2\sqrt{q}$.

Using the value of $\tau_1$, we determine the trace of $m$-th power of the chosen Frobenious element. For an abelian variety $A$, the action of $\theta_\mf{p}^m \curvearrowright V_\ell A$ has characteristic polynomial $P_m(T) \in \ZZ[T]$. When factoring over $\mathbb{C}$, $P_m(T) = (T-\alpha_1^m)(T-\alpha_2^m)$, where $-\tau_m$ is the coefficient of $T$ in $P_m(T)$. See \cite[\S 4.2]{quadfields} The values of $\alpha_1$ and $\alpha_2$ are determined by the trace of $\tau_1$ above. Hence, the trace of some $\theta_\mf{p}^m$ must be given by $\tau_m = \alpha_1^m + \alpha_2^m$. However, the $e$-th power gives us an extra congruence to consider. As $\tau_1$ must determine $\tau_e$, we still only consider about $4\sqrt{q}$ possible values for $\tau_e$ even as $e$ grows.
By \eqref{eq: ei =j}:
\begin{equation}\label{eq: te - q-q}
    \tau_e - q^{j_1} - q^{j_2} \equiv 0 \Mod{\ell}.
\end{equation}
In other words, for each $p$, if $\ell \neq p$, then $\ell\vert \left(\tau_e - p^{f\cdot j_1}-p^{f\cdot j_2}\right)$. There are only finitely many possible values for $\tau_e$ and $f$, and so only finitely many $\ell$ which satisfy \eqref{eq: te - q-q}. First screening against $p = 3$, we obtain a finite list of $\ell$ which can possibly admit a heavenly elliptic curve with Tate-Oort numbers $(j_1,j_2)$. We now repeat this screen for other $p \leq 11$, removing any $\ell$ which fails \eqref{eq: te - q-q} for every choice of $\tau_e$ and $f$.

For the next step of the computation, we consider possible options for the powers of $\chi$ that appear on the diagonal of the representation given in Proposition \ref{prop: heavenly upper triangular}. We now fix a prime $\ell$ from our reduced list of possible primes for our fixed Tate-Oort pair $(j_1, j_2)$. Recall that 
$$ ei_1 \equiv j_1 \Mod{\ell-1}\: \eqref{eq: ei =j} \qquad \text{and} \qquad i_2 +i_1 \equiv 1 \Mod{\ell-1} \: \eqref{eq: sum i-val}.$$

Once we solve for the values of $(i_1,i_2)$, it is possible to compute the trace of our representation $\rho_{E,\ell}$ with respect to a Frobenius element $\theta_\mathfrak{p}$. The pairs $(i_1, i_2)$ are determined only by the representation, so they must provide a valid trace for a Frobenius element at any prime $p$.

We consider each $(i_1, i_2)$ pair that meets these congruences. If $\mf{p}$ is a prime of $K$ of norm $q = p^f$ then each Forbenius element $\theta_\mf{p}$ must have $\chi^{i_1} (\theta_\mf{p})+ \chi^{i_2}(\theta_\mf{p}) \equiv \tau_1 \Mod{\ell}$. In particular, for each pair of values $(i_1,i_2)$, we check if there exists an integer representative of $\pm \left(p^{f\cdot i_1} + p^{f\cdot i_2}\right) \Mod{\ell}$ in the range $ \left[-\lfloor2\sqrt{p^f}\rfloor, \lfloor2\sqrt{p^f}\rfloor\right]$. If this condition fails for even a single $\mf{p}$, then $(i_1,i_2)$ cannot be the powers of $\chi$ in the representation. In some cases, every possible $(i_1,i_2)$ pair can be eliminated, and so a curve heavenly at $\ell$ with Tate-Oort numbers $(j_1,j_2)$ cannot exist.

%Finally, we can remove some remaining cases by appealing to \cite[Prop. 6.2, Cor. 6.4]{RT-constrainedtorsion}. This gives us the extra conditions that if our curve $E/K$ is heavenly at $\ell$ with $e = 8$ then $\ell \not\equiv 1 \Mod{4}$, if $e = 9$ then $\ell \not\equiv 1 \Mod{3}$, if $e = 12$ then $\ell \not\equiv 1 \Mod{4}$ or $\ell \not\equiv 1 \Mod{6}$, and finally if $e = 18$ then $\ell \not\equiv 1 \Mod{6}$. We may only used the larger $e$ values for this reduction as if it is possible that $e_\mathfrak{l} = 1$, we cannot guarantee that $f_\mathfrak{l} = 1$, and the lemma may not apply.

All the remaining possible cases are summarized in Table \ref{tab:survivingcase}.

\end{proof}

\section{Curves over Odd Degree Fields}\label{sec:odd-degree}

    \begin{proposition}\label{prop: balanced at rational}
    Let $E/K$ be an elliptic curve and suppose $K/\QQ$ is an extension of odd degree. Suppose $E$ is balanced at the odd rational prime $\ell$. Then $\ell \equiv 3 \Mod{4}$.
\end{proposition}

\begin{proof}
    If $\ell = 3$, the result immediately follows. So suppose $\ell > 3$ and $[K:\QQ] = n$. 
    As $E$ is balanced at $\ell$, $E$ is balanced at all primes $\mf{l}_i\vert\ell$. Suppose $\ell$ splits into $r$ primes. Then  $\sum_{i=1}^{r}f_{\mf{l}_i}e_{\mf{l}_i} = n$.

    It cannot be the case that every $e_{\mf{l}_i}$ is even since $n$ is odd.
    So up to re-ordering, we may assume that $e_{\mf{l}_1}$ is an odd integer.

    As $E$ is balanced at $\mf{l}_1$, $4\vert e_E(\mf{l}_1)$ and $\ord_2(e_E(\mf{l}_1)) > \ord_2(\ell-1)$ by \cite[Lem. 4.1]{quadfields}. By Proposition \ref{prop: semistable options}, $e_E(\mf{l}_1) \in \{e_{\mf{l}_1}, 2e_{\mf{l}_1}, 3e_{\mf{l}_1}, 4e_{\mf{l}_1}, 6e_{\mf{l}_1}\}$. 
    Hence $e_E(\mf{l}_1) = 4e_{\mf{l}_1}$ is the only possible solution.
    Thus $\ord_2(e_E(\mf{l}_1)) = 2 > \ord_2(\ell-1) = 1$. This is only possible if $\ell \equiv 3 \Mod{4}$.

\end{proof}

The main observation in this proof is that there must always be an odd ramification index for an odd degree extension. This directly implies a corollary in the case where the extension $K/\QQ$ is Galois, since any prime $\mf{l}$ above $\ell$ will have odd ramification index.

\begin{corollary}\label{cor: balanced Galois}
    Let $E/K$ be an elliptic curve and suppose $K/\QQ$ is a Galois extension of odd degree.  Suppose $E$ is balanced at some prime $\mf{l}$ above $\ell$, an odd prime. Then $\ell \equiv 3\Mod{4}$.
\end{corollary}

Since $E/K$ is balanced at $\mf{l}$, we know the Tate-Oort numbers at $\mf{l}$ are $j_1 = j_2 = \frac{e}{2}$. So \eqref{eq: ei =j} implies

\begin{equation}\label{eq: balanced 2ei = e}
    2ei_r \equiv e \Mod{\ell -1}, \qquad r = 1,2.
\end{equation}

This equation has $\mathrm{gcd}(2e,\ell-1)$ solutions. For any odd prime $\ell$, $2\vert (\ell-1)$ and hence $2\vert\mathrm{gcd}(2e,\ell-1)$. This means way may always partition the solutions to \eqref{eq: balanced 2ei = e} into pairs $\{i_1, i_2\} \Mod{\ell - 1}$. 

In the quadratic case there is always unique set of solutions $\{i_1, i_2\} \Mod{\ell-1}$, however that may not be the case when $[K:\QQ] > 2$. The next lemma shows that we may form the partition of solutions such that \eqref{eq: sum i-val} holds for each pair.

\begin{lemma}\label{lem: i1,i2 soln} Let $E/K$ be an elliptic curve heavenly at an odd prime $\ell$ and suppose $K/\QQ$ is an extension of odd degree. Further, suppose either that $E$ is balanced at $\ell$, or $K/\QQ$ is Galois and $E$ is balanced at some prime $\mf{l}$ above prime $\ell$. The solutions to \eqref{eq: balanced 2ei = e} are:

\begin{enumerate}
    \item[(a)] If $\ell = 3$, then $\{i_1,i_2\} \equiv \{0,1\} \Mod{2}$.

    \item[(b)] If $\ell > 3$, let $s \in \ZZ$ such that $2es \equiv 0 \Mod{\ell-1}$. Then $\{i_1,i_2\} \equiv \{\frac{\ell+1}{4} + s, \frac{3\ell-1}{4} - s\} \Mod{\ell -1}$.
\end{enumerate}
    
\end{lemma}

 \begin{proof}
    The solution for $(a)$ is clear, since $4\vert e$ \cite[Lem. 4.1]{quadfields}. In this case \eqref{eq: balanced 2ei = e} reduces to $0i_r \equiv 0 \Mod{\ell-1}$.

    For part $(b)$, it is easy to check that $\{(\ell+1)/4, (3\ell-1)/4\} \Mod{\ell -1}$ gives a valid set of solutions. Here $s$ can be thought of as a shift of this initial solution. Shifting in this way maintains the constraint from \eqref{eq: sum i-val}, which in the elliptic curve case simply says $i_1 + i_2 \equiv 1 \Mod{\ell -1}$.

    Let $s \in \ZZ$ such that $2es \equiv 0 \Mod{\ell -1}$. We will check that $((\ell+1)/4) +s$ provides a solution to \eqref{eq: balanced 2ei = e}. The case with $((3\ell-1)/4) -s$ works similarly.

    $$2e\left(\frac{\ell+1}{4} +s\right) \equiv 2e\left(\frac{\ell+1}{4}\right) + 2es
        \equiv e + 2es \equiv e \Mod{\ell-1} $$
\end{proof}

\begin{remark}
    Suppose we have $2n$ solutions to \eqref{eq: balanced 2ei = e}, where $n$ is an odd divisor of $[K:\QQ]$. This indicates there are $n$ pairs of solutions, one of which must be $((\ell+1)/4, (3\ell-1)/4)$. In which case we have $n-1$ shifted solutions, and hence $(n-1)/2$ unique shift values $s \Mod{\ell -1}$ by the previous proposition.

   % This also indicates what particular values $s$ can take. By assumption $\mathrm{gcd}(2e,\ell-1) = 2k$ with $k > 1$. In order for $2es \equiv 0 \Mod{\ell-1}$, it must be that $2es\vert (\ell-1)$ and hence $\left(\frac{\ell-1}{k}\right)\vert s$. 
   %Thes remark became redundant with the new lemma obove, but i wanted to keep the divisibility note here.
\end{remark}

We call a, necessarily heavenly, elliptic curve \textit{totally balanced} if it is balanced at some prime $\mf{l}\vert\ell$ and $\{i_1,i_2\} \equiv \{(\ell+1)/4, (3\ell -1)/4\} \Mod{\ell-1}$.

\begin{remark}
    For a given balanced curve $E/K$, if $\mathrm{gcd}(2e,\ell-1) = 2$, then $\{(\ell+1)/4, (3\ell -1)/4\}$ is the unique set of solutions $\Mod{\ell -1}$ and $E$ must be totally balanced. First, $\mathrm{gcd}(2e,\ell-1) = 2$ clearly holds if $e_\mf{l} = 1$. Next suppose $p$ is some prime dividing $e_\mf{l}$. Then $p\vert (\ell -1) \iff \ell \equiv 1 \Mod{p}$. So if it is the case that $\ell \not\equiv 1 \Mod{p}$ for all $p\vert [K:\QQ]$ then $\mathrm{gcd}(2e,\ell-1) = 2$.
\end{remark}

\begin{proposition}\label{prop: new trace}
Suppose $K/\QQ$ is an odd degree extension, and $E/K$ is heavenly at an odd prime $\ell$. Suppose either that $E$ is balanced at $\ell$, or $K/\QQ$ is Galois and $E$ is balanced at some prime $\mf{l}$ above $\ell$. Further suppose $\mf{p}$ is a prime of $K$ above $p \neq \ell$, with residue degree $f \defeq f_\mf{p}$ and norm $q = \textbf{N}\mf{p}$. Then 

\begin{equation}
    a_\mf{p}(E) \equiv q^{(\ell +1)/4}\left(q^s + q^{-s}\left(\frac{p}{\ell}\right)^f\right) \Mod{\ell}.
\end{equation}

In particular when E is totally balanced, we have

\begin{equation}\label{eq: totally balanced}
    a_\mf{p}(E) \equiv q^{(\ell +1)/4}\left(1 +\left(\frac{p}{\ell}\right)^f\right) \Mod{\ell}.
\end{equation}

\end{proposition}

\begin{proof}
Let $\theta_\mf{p} \in G_K$ be a Frobenius element for $\mf{p}$. Since $\chi(\theta_\mf{p}) = q$, 

\begin{align*}
    a_\mf{p}(E) \equiv \trace(\rho_{E,\ell}(\theta_\mf{p})) & \equiv q^{i_1} + q^{i_2}\\
    & \equiv q^{((\ell +1)/4) +s} + q^{((3\ell-1)/4) - s}\\
    & \equiv q^{(\ell + 1)/4}\left(q^s + q^{((\ell-1)/2) -s} \right)\\
    & \equiv q^{(\ell+1)/4}\left(q^s + q^{-s}\left(\frac{p}{\ell}\right)^f\right) \Mod{\ell}.
\end{align*}

We deduce the equation for totally balanced curves by simply setting $s = 0$.
\end{proof}

\begin{proposition} \label{Prop: balanced trace KL}
    Let $K/\QQ$ be a Galois extension of odd degree and $E/K$ be an elliptic curve totally balanced at $\mf{l}$, a prime of $K$ above an odd prime $\ell$. Let $\mf{p}$ be a prime of $K$ above $p$, a prime distinct from $\ell$. Let $f: = f_{\mf{p}}$, and $L = \QQ(\sqrt{-\ell})$.
    \begin{enumerate}
        \item[(a)] If $\mf{p}$ is inert in $KL$, $a_\mf{p}(E) \equiv 0 \Mod{\ell}$.

        \item[(b)] If $\mf{p}$ splits in $KL$, $a_\mf{p}(E)^2 \equiv 4p^f \Mod{\ell}$.

    \end{enumerate}
\end{proposition}

\begin{proof}
    Let $\mf{P}$ be a prime of $KL$ above $\mf{p}$, and $\mf{q} = \mf{P}\cap \OO_L$.
    By Corollary \ref{cor: balanced Galois} $\ell \equiv 3 \Mod{4}$, so we have $\left(\frac{-\ell}{p}\right) = \left(\frac{p}{\ell}\right)$. Thus $p$ splits in $L$, meaning $f_\mf{q} = 1$, if and only if $\left(\frac{p}{\ell}\right) = 1$.

    As both $K$ and $L$ are Galois over $\QQ$, the extensions $KL/\QQ$, $KL/L$, and $KL/K$ are all be Galois as well \cite[\S 14.4 Prop. 21]{dummit2003abstract}.

    For $(a)$, suppose $\mf{p}$ is inert in $KL$. Then $f_{\mf{P}|\mf{p}}$ = 2. Since $K/\QQ$ is a Galois extension of odd degree, it must be that $f$ is some odd divisor of the degree of the extension. We can now compare inertial degrees using our two distinct multiplicative towers of fields.

    \[f_\mf{P} = f_{\mf{P}|\mf{p}}\cdot f = 2f = f_{\mf{P}|\mf{q}}\cdot f_\mf{q}.\]

    As $KL/L$ is Galois of odd degree, $f_{\mf{P}|\mf{q}}$ is odd. Hence $f_\mf{q} = 2$ and $p$ is inert in $L$, so $\left(\frac{p}{\ell}\right) = (\frac{p}{\ell})^f = -1$.

    By Proposition \ref{prop: new trace}, $a_\mf{p}(E) \equiv 0  \Mod{\ell}.$

    For $(b)$, suppose $\mf{p}$ splits in $KL$. Then $f_{\mf{P}|\mf{p}} = 1$, $f_{\mf{P}}$ is odd, and $f_\mf{q} = 1$. Since $\ell \equiv 3 \Mod{4}$, $\ell$ is the only prime that ramifies in $L$. Thus $p$ splits in $L$, i.e. ,$(\frac{p}{\ell}) = 1$. Proposition \ref{prop: new trace} now implies 
    \[a_\mf{p}(E)^2 \equiv 4\left(p^f\right)^{(\ell + 1)/2} \equiv 4\left(p^{(\ell + 1)/2}\right)^f \equiv 4\left(p\left(\frac{p}{\ell}\right)\right)^f \equiv 4p^f \left(\frac{p}{\ell}\right)^f \equiv 4p^f \Mod{\ell}.\qedhere \]

    \end{proof}

\begin{proposition}\label{prop: heavenly not surj}
Let $K/\QQ$ be a Galois extension of odd degree and suppose $E/K$ is an elliptic curve totally balanced at a prime $\mf{l}$ above an odd prime $\ell$. Then 
\[\tr\left(\rho_{E,\ell}(G_K)\right) \subseteq \left(\frac{2}{\ell}\right)\cdot\FF_\ell^{\times2}\cup \{0\}.\]
    
\end{proposition}

\begin{proof}
    In order to show the image of the trace is non-surjective, we first recall that $\ell \equiv 3\: (\text{mod } 4)$, and so $-1$ is not a square $(\text{mod } \ell)$. 
    Hence $\left(\frac{2}{\ell}\right)\cdot\FF_\ell^{\times2}$ contains exactly the square classes or non-square classes depending on the sign of $\left(\frac{2}{\ell}\right)$. By the Chebotarev Density Theorem and the finiteness of $\rho_{E,\ell}$ in every coset of $\ker(\rho_{E,\ell})$ is represented by a Frobenius element $\theta_\mf{p}$ for some $\mf{p}\nmid \ell$.
    So we must show that for all $a_\mf{p} \in \tr\left(\rho_{E,\ell}(G_K)\right)$ it is the case that either $a_\mf{p} \equiv 0 \Mod{\ell}$ or $\left(\frac{a_\mf{p}}{\ell}\right) = \left(\frac{2}{\ell}\right).$

    First set $f\defeq f_\mf{p}$. Suppose that $\left(\frac{p}{\ell}\right) = 1$. Then \eqref{eq: totally balanced} implies $a_\mf{p} \equiv 2\left(p^f\right)^{\frac{f(\ell + 1)}{4}}\: (\text{mod } \ell)$. 
    Now $\left(\frac{a_\mf{p}}{\ell}\right) = \left(\frac{2p^{\frac{f(\ell+1)}{4}}}{\ell} \right) = \left(\frac{2}{\ell}\right) \left(\frac{p}{\ell}\right)^\frac{f(\ell+1)}{4} = \left(\frac{2}{\ell}\right)$ as desired.

    Now suppose $\left(\frac{p}{\ell}\right) = -1$. Then $a_\mf{p} \equiv \left(p^f\right)^{\frac{\ell+1}{4}}(1 + (-1)^f)\: (\text{mod } \ell)$. 
    As $K/\QQ$ is a Galois extension of odd degree, $f$ must be odd. 
    Thus $a_\mf{p} \equiv 0\: (\text{mod } \ell)$.
\end{proof}

\begin{remark}
    In the quadratic case, \cite{quadfields}, the authors were able to show that $\tr\left(\rho_{E,\ell}(G_K)\right) = \left(\frac{2}{\ell}\right)\cdot\FF_\ell^{\times2}\cup \{0\}$. It is a surprising result that the reverse containment was not immediate over cubic fields. We will attempt to classify a partial result where we can show the trace of $\rho_{E,\ell}$ is \textit{precisely} the squares or non-squares inside $\FF_\ell^\times$.
\end{remark}

    For Lemma \ref{lemma: intersection square class}, we make the mild assumption that $KL\cap \QQ(\bmu_\ell) = L$, where $L = \QQ(\sqrt{-\ell})$. Here we classify two sufficient conditions for this to hold.

\begin{lemma}\label{lem: when KL cap Q = L}
    Let $K/\QQ$ be Galois of odd degree, $\ell$ an odd prime, and $L = \QQ(\sqrt{-\ell})$, $KL \cap \QQ(\bmu_\ell) = L$ if either of the following hold:
    \begin{enumerate}
        \item[(a)] $K$ is a real number field.
        \item[(b)] For $\mf{l}$ a prime of $K$ above $\ell$ $\left(e_\mf{l},(\ell-1)/2\right) = 1$.
    \end{enumerate}
\end{lemma}

\begin{proof}
    The proof for $(a)$ is immediate. We classify this result as if $E$ has complex multiplication then $K = \QQ(j(E))$ is a real number field.

    For $(b)$ as $K/\QQ$ is Galois of odd degree $e_\mf{l}$ is odd. As $\ell$ must ramify in $L$, we know the ramification index of $\ell$ in $KL$ is $2\cdot e_\mf{l}$. Let $L' := KL\cap\QQ(\bmu_\ell)$, we will calculate the degree of $[L':L]$.

    As $L'$ is a subfield of $KL$, we know the ramification index of a prime above $\ell$ in $L'/L$ must a divisor of $e_\mf{l}$. $L'$ must also be a subfield of $\QQ(\bmu_\ell)$ so $[L':L]\Big\vert \displaystyle\frac{\ell-1}{2}$. In addition, such an extension must be completely ramified at $\ell$, meaning the ramification index is the degree of the extension. Hence we can conclude $[L':L] \Big\vert e_\mf{l}$ as well. Therefore as $\left(e_\mf{l}, (\ell-1)/2\right) = 1$ by assumption, we must have that $L' = L$.
\end{proof}

\begin{remark}\label{rmk: KL cap Q condition totally balanced}
Note that $(e_\mf{l}, (\ell-1)/2) =1$ is equivalent to $(2e_\mf{l}, \ell-1) = 2$. So if $E/K$ is an elliptic curve which is balanced at $\mf{l}$ and condition $(b)$ holds, then $E/K$ is totally balanced at $\mf{l}$.
    
\end{remark}

\begin{lemma}\label{lemma: intersection square class}
     Let $K$ and $L$ as in Proposition \ref{Prop: balanced trace KL} and suppose $KL\cap \QQ(\bmu_\ell) = L$. For any fixed $b \in \FF_\ell^\times$, there exists a prime $p$ that splits completely in $KL$ such that $p \equiv b^2 $ (mod $\ell$). 
\end{lemma}

\begin{proof}
    We'll consider the following diagram of field extensions. 

    
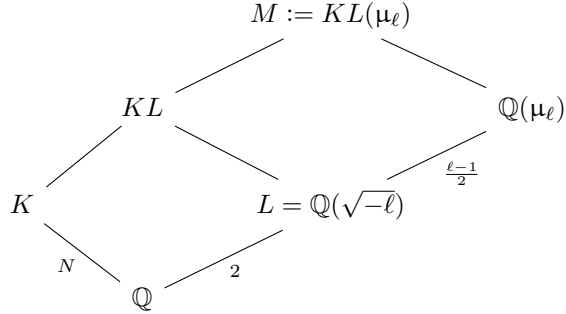
\begin{center}
\begin{tikzcd}
                      &                                                               & M \defeq KL(\bmu_\ell)                                                                        &                                   \\
                      & KL \arrow[ru, no head]                                                 &                                                                                           & \mathbb{Q}(\bmu_\ell) \arrow[lu, no head] \\
K \arrow[ru, no head] &                                                               & L = \mathbb{Q}(\sqrt{-\ell}) \arrow[lu, no head] \arrow[ru, "\frac{\ell-1}{2}"', no head] &                                   \\
                      & \mathbb{Q} \arrow[lu, "N", no head] \arrow[ru, "2"', no head] &                                                                                           &                                  
\end{tikzcd}
    \end{center}
    Let $b \in \FF_\ell^\times$ be given. Our goal is to find a Frobenius element in $\Gal(M/\QQ)$ corresponding to a rational prime that has the desired properties. For any number field, there are infinitely many primes that split completely, however our hope is to justify that there must always exist one that is congruent to $b^2$ for any fixed $b$. This will be done by first considering elements of $\Gal(KL/\QQ) \times \Gal(\QQ(\bmu_\ell)/\QQ)$. We then apply the Chebortarev Denisty Theorem to lift to a suitable element in $G_\QQ$. Note
     $$\Gal(M/\QQ) \cong \{(\sigma, \tau) : \sigma\vert_L = \tau\vert_L\} \leq \Gal(KL/\QQ) \times \Gal(\QQ(\bmu_\ell)/\QQ),$$ so we may fix a homomorphism $\phi:  \Gal(KL/\QQ) \times \Gal(\QQ(\bmu_\ell)/\QQ) \to \Gal(M/\QQ) $ \cite{dummit2003abstract}.

    First consider the criteria that $p$ slits completely in $KL$ and let $\theta_p \in G_\QQ$ be a Frobenius element for $p$. By the proof of Proposition \ref{Prop: balanced trace KL}, $\left(\frac{p}{\ell}\right) = 1$. Now consider $\chi$, the $\ell$-adic cyclotomic character modulo $\ell$. 
   
    This map has the property that $\chi(\theta_p) \equiv p $ (mod $\ell$). The second criteria we ask is that $p$ is a square $\Mod{\ell}$. For such a prime $p$ it must be that  $\chi(\theta_p) \in \Gal(\QQ(\bmu_\ell)/L)$. So we conclude that any Frobenius corresponding to a prime that splits completely in $KL$ and is congruent to a square modulo $\ell$ must fix the subfield $L$.

    Hence we can now reduce our search to the following subgroup
    \[H \defeq \Gal(KL/L) \times \Gal(\QQ(\bmu_\ell)/L).\]

    As $KL \cap \QQ(\bmu_\ell) = L$ by assumption, $\Gal(M/L) \cong H$. Hence the restriction of $\phi$ to $H$, $\phi_H : H \to \Gal(M/L)$, is an isomorphism.

    Hence for any $\tau \in \Gal(\QQ(\bmu_\ell)/L)$ representing a specific square modulo $\ell$, there exists $\sigma \in \Gal(KL/L)$ such that $\phi(\sigma, \tau ) \in \Gal(M/L) \leq \Gal(M/\QQ)$ as desired. \end{proof}

This lemma is is needed to show that in the special case where $KL \cap \QQ(\bmu_\ell) = L$, any member of $\displaystyle\left(\frac{2}{\ell}\right)\cdot\FF_\ell^{\times2} \cup \{0\}$ can be written as the trace of Frobenius for some prime $\mf{p}$ in $K$.

\begin{corollary}\label{Cor: reverse containment}
    With $K,L, E$ as in Proposition \ref{prop: heavenly not surj} and Lemma \ref{lemma: intersection square class}, then
    \[\tr\left(\rho_{E,\ell}(G_K)\right) = \left(\frac{2}{\ell}\right)\cdot\FF_\ell^{\times2}\cup \{0\}.\]
\end{corollary}

\begin{proof}
    It is only necessary to show the reverse of the containment established in Proposition \ref{prop: heavenly not surj}. 

    Let $a \in \left(\frac{2}{\ell}\right)\cdot\FF_\ell^{\times2}\cup \{0\}$. We must demonstrate a prime $\mf{p}$ such that $a_\mf{p} \in \tr\left(\rho_{E,\ell}(G_K)\right)$ satisfies $a_\mf{p} \equiv a \Mod{\ell}$.

    If $a = 0$, pick any prime $p$ inert in $L$, and choose any prime $\mf{p}$ of K above $p$. Then $\left(\frac{p}{\ell}\right) = -1$, and so $a_\mf{p} \equiv 0$ as $f_\mf{p}$ must always be odd.

    Now suppose $a \neq 0$, and set $b = \frac{a}{2} \in \FF_\ell^\times$. By Lemma \ref{lemma: intersection square class}, there exists a prime $p$ that splits completely in $KL$ and has the property that $p \equiv b^2 \equiv \left(\frac{a}{2}\right)^2$.

  As $p$ splits completely in $KL$, it must also split completely in all sub-extensions. So for some prime $\mf{p}$ of $K$  above $p$, it must be that $\mf{p}$ splits in $KL$. By Proposition \ref{Prop: balanced trace KL}, $$a_\mf{p}(E)^2 \equiv 4p^f \equiv (2b)^2 \equiv a^2\: (\text{mod } \ell).$$ So it must be that $a_\mf{p}(E) \equiv \pm a \Mod{\ell}$. We have already shown $a_\mf{p} \in \left(\frac{2}{\ell}\right)\cdot\FF_\ell^{\times2}\cup \{0\}$ by the first containment, and $a \in \left(\frac{2}{\ell}\right)\cdot\FF_\ell^{\times2}\cup \{0\}$ by assumption.  So we must now simply show that $-a \not\in \left(\frac{2}{\ell}\right)\cdot\FF_\ell^{\times2}\cup \{0\}$. However $\ell \equiv 3 \Mod{4}$ and so $-1 \not\in \FF_\ell^{\times2}$. Thus $-a \notin \left(\frac{2}{\ell}\right)\FF_\ell^{\times 2}$ also.

\end{proof}

\section{Connections to Complex Multiplication}\label{sec:CM-heavenly}

    The trace behavior of a totally balanced elliptic curve over a Galois field very closely mirrors the behavior of an elliptic curve with complex multiplication. We will catalog some results simply about CM curves, and then determine exactly when a CM curve is also heavenly. We would like to note here that when we start by assuming our curve has CM, we do not need to assume the field of definition is Galois. The tools from class field theory allow us to restrict the behavior of primes without adding this assumption. Removing the Galois assumption means many of the results in this section are direct generalizations of work in \cite{quadfields}, however the usage of relative discriminants in Proposition \ref{prop: CM good red l 3} is new.

\begin{proposition}\label{prop: CM good red l 3}
    Suppose $E/K$ is an elliptic curve defined over $K = \QQ(j(E))$, an odd degree extension of $\QQ$.
    Suppose further that $E$ has complex multiplication by the maximal order of an imaginary quadratic field $L$. 
    If $E$ has good reduction away from an odd prime $\ell$, then $L = \QQ(\sqrt{-\ell})$ and $\ell \equiv 3 \Mod{4}$. 
\end{proposition}

\begin{proof}
    Set $n = [K:\QQ]$, and consider the composite field $KL = L(j(E))$.
    This must be the Hilbert class field of $L$. \cite[Ch. II.4, Thm. 4.3]{advancedSilverman}
    Furthermore, since $n$ is be odd, we have $K \cap L = \QQ$ and hence $[KL:L] = [K:\QQ]$. We will denote the genus field of $L$ by $\Gamma(L)$, where $\Gamma(L) \subseteq KL$.
    
    By the classification of genus fields given in \cite[Thm. 6.1]{primesx+ny}, $\Gamma(L) = L(\sqrt{p_1^*}, \ldots , \sqrt{p_r^*})$ with each $p_i$ a factor of $\Delta_{L/\QQ}$ and $p_i^* = (\frac{p_i}{4})\cdot p_i$.
    Hence the degree of the extension $\Gamma(L)/L$ must be a power of 2.
    As $KL/L$ has odd degree, $L = \Gamma(L)$. 
    Hence $\Delta_{L/\QQ}$ must have only one prime factor, say $\Delta_{L/\QQ} = p$. 
    For this to be true for an imaginary quadratic field, $L$ must have the form $\QQ(\sqrt{-p})$ with $p \equiv 3\Mod{4}$.

    We now endeavor to show that $ p = \ell$.
    First, since $E/K$ has good reduction away from $\ell$, the extension $K(E[\ell])/K$ must be unramified away from $\ell$.
    By \cite[Lem. 3.15]{BCS17}, $L \subseteq K(E[\ell])$.
    So we have the following tower of field extensions on which we may compare discriminants.

    
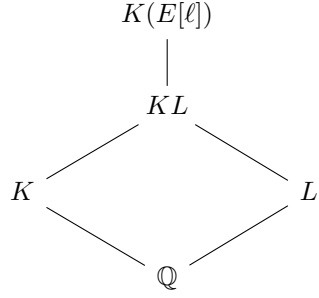
\begin{center}
        % https://tikzcd.yichuanshen.de/#N4Igdg9gJgpgziAXAbVABwnAlgFyxMJZARgBoAGAXVJADcBDAGwFcYkQBpACgFFkAdfjEaNKAShABfUuky58hFGWLU6TVuw4AZKTJAZseAkXKkATKoYs2iTrtmGFRM+cvqbIHdIfzjS0gDMbtbsggC29DgAFgBGMcAAipJSqjBQAObwRKAAZgBOEGFIACw0OBBIATRRMPRQ7JBgbN4g+YUlZRWILiA1dQ0EzXptRd2dSGS9tfW2jUO5BaOT5UimU-2zg-ati5XjiJN9M+BbkpSSQA
\begin{tikzcd}
                      & {K(E[\ell])}                                       &                       \\
                      & KL \arrow[u, no head]                              &                       \\
K \arrow[ru, no head] &                                                    & L \arrow[lu, no head] \\
                      & \mathbb{Q} \arrow[ru, no head] \arrow[lu, no head] &                      
\end{tikzcd}
    \end{center}

We may compute the discriminant $\Delta_{KL/\QQ}$ by comparing relative discriminants down either side of the tower. By \cite[Cor. 2.10]{algebraic_number_neukrich} we find a formula that relates absolute and relative field discriminants.  

\begin{equation}\label{eq: relative disc.}
    \Delta_{KL/\QQ} = \Delta_{K/\QQ}^{[KL:K]}\cdot N_{K/\QQ}(\Delta_{KL/K}) = \Delta_{L/\QQ}^{[KL:L]}\cdot N_{L/\QQ}(\Delta_{KL/L}).
\end{equation}

 Further we know that a a prime $\mf{p} \in K$ ramifies in $KL$ if and only if $\mf{p}\big\vert\Delta_{KL/K}$. As $KL/K$ is an extension unramified away from $\ell$, it must be that the only factors of the discriminant $\Delta_{KL/K}$ are primes $\mf{l}\big\vert \ell$. Hence, $N_{K/\QQ}(\Delta_{KL/K}) = \ell ^r$ for some $r$. Similarly, $KL/L$ is a totally unramified extension, so its discriminant $\Delta_{KL/L}$ must be a unit in $\OO_L$ and so $N_{L/\QQ}(\Delta_{KL/L}) = \pm 1$. Finally $\Delta_{L/\QQ} = -p$. Then \eqref{eq: relative disc.} simplifies to

 \[\Delta_{KL/\QQ} = (\Delta_{K/\QQ})^2\cdot \ell^r = \pm(-p)^n.\]

 It is clear that the right hand side of our equality has only one prime factor, and hence $p = \ell$.
\end{proof}

\begin{remark}
    Under the hypothesis of Proposition \ref{prop: CM good red l 3}, we have $KL/\QQ$, and hence $K/\QQ$ unramified away from $\ell$. This greatly restricts the fields that can appear as the odd-degree field of definition of an elliptic curve with complex multiplication, and good reduction away from one rational prime.
\end{remark}

\begin{remark}
      Note that $\sharp\textit{Cl}(\OO_L) = [KL:K]$ is a function of $n = [K:\QQ]$ only. There are finitely many imaginary quadratic fields with a fixed class number, so for each $n$ only finitely many choices of $\OO_L$ are possible. \cite{fixed-class-number}. This finite list of maximal orders may only produce examples of heavenly elliptic curves that meet the hypothesis of Proposition \ref{prop: CM good red l 3}. 
\end{remark}

The following two corollaries were also shown for $n = 2$ in \cite{quadfields}.

\begin{corollary}\label{cor: CM not everywhere good}
    An elliptic curve $E/K$ over a field $K = \QQ(j(E))$ of odd degree with complex multiplication by a maximal order $\OO_L$ cannot have everywhere good reduction.
\end{corollary}

\begin{proof}
    An elliptic curve with everywhere good reduction will have good reduction away from $\ell$ for every prime. Hence we must simply apply the previous theorem with some $\ell' \equiv 1 \Mod{4}$. We then conclude that $\ell' \equiv 3 \Mod{4}$, which is clearly a contradiction.
\end{proof}

\begin{corollary}\label{cor: CM not heavely twice}
    An elliptic curve $E/K$ over a field $K = \QQ(j(E))$ of odd degree with complex multiplication by a maximal order $\OO_L$ cannot be heavenly at two distinct primes.
\end{corollary}

\begin{proof}
    If such a curve were to be heavenly at multiple primes, $\ell $ and $\ell'$, this implies that the curve must have good reduction away from $\ell$ and good reduction away from $\ell'$. Such a curve must then have good reduction everywhere, and we may apply Corollary \ref{cor: CM not everywhere good}.
\end{proof}

We now use the Hecke character to describe the trace of Frobenius for CM elliptic curves with good reduction away from a specific prime. We will assume the hypotheses of Theorem \ref{prop: CM good red l 3}. Let $M \defeq KL$ and let $E_M \defeq E\times_K M$. Let $\mf{P}$ be a prime of $M$ above $p$, and let $\mf{p} = \mf{P}\cap\OK$ and $\mf{q} = \mf{P}\cap\OO_L$ be the corresponding primes of $K$ and $L$ respectively.

We recall the definition of the Gr{\"o}\ss encharacter and some of its properties. This is a map $\psi_{E/M}: \mathbb{A}_M^\times \to \mathbb{C}^\times$ which carries information about the point counts for the reduction of $E$ at various primes. For a fractional ideal $\mf{P}$ of $M$, $\psi_{E/M}(\mf{P}) : = \psi_{E/M}(\ldots , 1, 1, \pi, 1, 1, \ldots)$ where $\pi$ is a uniformizer at $\mf{P}$, and appears in the $\mf{P}$ component. This map is well-defined independent of the choice of uniformizer exactly when $E$ has good reduction at $\mf{P}$ \cite[Ch. II.10]{advancedSilverman}.

In addition, $ \xi \defeq \psi_{E/M}(\mf{P}) \in L^\times$ generates the same principal fractional ideal in $L^\times$ as $N_{M/L}(\mf{P})$. Since $\mf{P}$ is an integral ideal, $\xi \in \OO_L$.\cite[Ch. II.9]{advancedSilverman}. So $\xi = (u  + v\sqrt{-\ell})/2$ for some $u,v\in \mathbb{Z}$.

%Had a citation from Advanced Silverman here that I removed. The fact that xi generates the same ideal as N_{M/L}(\mf{P}) (and the fact that the reduction of the map modulo p is the p-th frobenius) immediately imply these facts about the trace and norm.

The trace of $\xi$ in $L/\QQ$ gives $a_\mf{P}(E_M)$ and the norm of $\xi$ in $L/\QQ$ gives the size of the residue field $p^{f_\mf{P}}$. In particular,

\[a_{\mf{P}}(E_M) = \frac{u + v\sqrt{-\ell}}{2} + \frac{u - v\sqrt{-\ell}}{2} = u,\]

\[p^{f_\mf{P}} = N_{M/\QQ}\left(\mf{P}\right) = N_{L/\QQ}\left(\frac{u + v\sqrt{-\ell}}{2}\right) = \frac{u^2 + \ell v^2}{4}.\]

Consequently,

\begin{equation}\label{eq:1}
    4p^{f_\mf{P}} = a_{\mf{P}}(E_M)^2 + \ell v^2.
\end{equation}

\begin{proposition}
    Let $E/K$ be an elliptic curve over a field $K = \QQ(j(E))$ of odd degree. Let $L = \QQ(\sqrt{-\ell})$, and suppose $E$ has complex multiplication by $\OO_L$ and good reduction away from $\ell$. Let $\mf{p}$ be a prime of $K$ above $p$, a prime distinct from $\ell$, and denote the residue field degree $f := f_\mf{p}$. 
    \begin{enumerate}
        \item[(a)] If $\mf{p}$ is inert in $KL$, then $a_\mf{p}(E) = 0$.
        \item[(b)] If $\mf{p}$ splits in $KL$, then $a_p(E)^2\equiv 4p^f \Mod{\ell}$.
    \end{enumerate}
\end{proposition}

\begin{proof}
    The proof of part (a) is Exercise 2.30 in \cite{advancedSilverman}. A longer description is given in \cite{OHara-thesis}.

    For (b), suppose $\mf{p}$ splits in $M$. Then it must be that the residue fields $\OO_M/\mf{P}$ and $\OK/\mf{p}$ coincide. Since $f = f_\mf{P}$ clearly the reduction of $E$ at $\mf{p}$ and $E_M$ at $\mf{P}$ have the same number of points as $\OK/\mf{p} \cong \OO_M/\mf{P}$. Thus $a_\mf{p}(E) = a_\mf{P}(E_M)$. By \eqref{eq:1}, $4p^f = a_\mf{p}(E)^2 + \ell v^2$ and hence $4p^f \equiv a_\mf{p}(E)^2 \Mod{\ell}$.
\end{proof}

\begin{theorem}\label{thm: CMheavenly iff not sur}
    Let $\ell$ be an odd prime, and $K$ an extension of $\QQ$ of odd degree. Suppose $E/K$ is an elliptic curve that satisfies $K = \QQ(j(E))$. Further, assume that $E/K$ has complex multiplication and has good reduction away from $\ell$.  Then $E/K$ is heavenly at $\ell$ if and only if $\trace(\rho_{E,\ell}(G_K))$ is a proper subset of $\FF_\ell$.
    
\end{theorem}

We now employ the same strategy used in \cite[Theorem 7.1]{quadfields}, which extends to the present case. For the convenience of the reader, we will give the argument in full.

\begin{proof}
    As $E$ has good reduction away from $\ell$, we must only show that $[K(E[\ell]):K(\mu_\ell)]$ a power of $\ell$. 
    Let $L$ be the field of complex multiplication, and set $M\defeq KL$. 
    From Remark \ref{rmk: pick maxl CM}, we may assume $E$ has CM by the maximal order $\OO_L$.
    By \cite{BCS17}, $L \subseteq K(E[\ell])$.
    Hence $M(E[\ell]) = K(E[\ell])$. 
    The $j$-invariant for any elliptic curve with complex multiplication must be real, hence $K$ is a real field, and $[K(\mu_\ell) : K] = \ell-1$. 
    By Theorem \ref{prop: CM good red l 3}, $L = \QQ(\sqrt{-\ell})$ with $\ell \equiv 3 \Mod{4}$.
    It then must be true that $M = K(\sqrt{-\ell}) \subseteq K(\mu_\ell) = M(\mu_\ell)$ and $[M(\mu_\ell):M] = (\ell-1)/2$.

    Let $\rho = \rho_{E,\ell}$ be the $G_K$-representation on $E[\ell]$. 
    As $E$ has complex multiplication, $E[\ell]$ is a free rank one module over $\OO_L/\ell\OO_L$.
    Since $\ell$ ramifies in $L$, $\OO_L/\ell\OO_L \cong \FF_\ell [T]/(T^2)$. 
    Then
    \[\rho(G_K) \leq \Aut_{\OO_L/\ell\OO_L}\left(E[\ell]\right) \cong \left(\FF_\ell[T]/\left(T^2\right)\right)^\times = \{bT + a : b \in \FF_\ell, a \in \FF_\ell^\times\}.\]

    By definition $\ker(\rho) = G_{K(E[\ell])}$.
    Since $E$ has good reduction away from $\ell$ and complex multiplication there is an injection
    \[\overline{\rho} : \Gal(K(E[\ell])/M) \hookrightarrow \Aut_{\OO_L/\ell\OO_L}(E[\ell]).\]

    As $K \subseteq M$, then $G_M \leq G_K$ and $\rho(G_M) \subseteq \rho(G_K)$. Hence we can conclude that the image of $G_{M}$ under $\rho$ must also sit inside of $\Aut_{\OO_L/\ell\OO_L}(E[\ell])$. 
    We may then choose an ordered basis for $E[\ell]$ such that
    \[\rho(G_{M}) \leq \mathcal{C} := \left\{\begin{pmatrix}
        a & b\\ 0 & a
    \end{pmatrix} : b \in \FF_\ell, a \in \FF_\ell^\times \right\}.\]\cite[Thm. 1.1]{Lozano_Robledo_2022}
    Let $\mathcal{N}$ be the normalizer of $\mathcal{C}$ inside $\GL_2(\FF_\ell)$. 
    There exists a unique subgroup $\mathcal{N}'\leq \mathcal{N}$ with the property that $[\mathcal{N}':\mathcal{C}] = 2$.
    This subgroup has the form
    \[\mathcal{N}' : = \left\{\begin{pmatrix}
        \pm a & b\\ 0 & a
    \end{pmatrix} : b \in \FF_\ell, a \in \FF_\ell^\times \right\}.\]
    Further, by \cite[Lem. 6.2]{Lozano_Robledo_2022}, it must be that $\rho(G_K) \leq \mathcal{N}'$.
    Consequently, we can express $\rho$ using homomorphisms $\theta: G_K \to \FF_\ell^\times$ and $\varepsilon : G_K \to \{\pm 1\}$ such that
    \[\rho = \begin{pmatrix}
        \varepsilon\theta & \star \\ & \theta \end{pmatrix}, \quad \rho\big\vert_{G_{M}} = \begin{pmatrix}
            \psi & \star\\ & \psi \end{pmatrix}, \ \text{where}\ \psi : = \theta\big\vert_{G_{M}}.\]
    Further, these must satisfy the condition $\varepsilon\theta^2 = \det\rho = \chi$.
    Notice now that $\varepsilon$ is a cokernel character for the projection $G_K \twoheadrightarrow G_K/G_M$ as for any $\sigma \in G_K$, we have 
    \[\trace\rho(\sigma) = 0 \iff \varepsilon(\sigma) = -1 \iff \sigma \notin G_{M}.\]
    So $\trace\left(\rho(G_K)\right) = \trace\left(\rho(G_{M})\right)\cup \{0\}$.
    Then for any $\sigma \in G_{M}$, $\trace\rho(\sigma) = 2\cdot\psi(\sigma)$.
    But as 2 is invertible in $\FF_\ell$,
    \[\sharp(\trace\rho(G_{M})) = \sharp(2\cdot\psi(G_{M})) = \sharp\psi(G_{M}).\]
    We now check if $\psi$ is surjective.
    Set $M^\psi \defeq \overline{M}^{\ker\psi}$. 
    The condition $\psi^2 = \chi\big\vert_{G_{M}}$ gives the following inclusion of Galois groups:
    \[G_{M^\psi} = \ker\psi \leq \ker\psi^2 = \ker\chi\big\vert_{G_{M}} = G_{M(\mu_\ell)}.\]
    Thus $M(\mu_\ell) \subseteq M^\psi$. The representation $\rho\big\vert_{M^{\psi}}$ is unit upper triangular. 
    So the image $\rho(M^\psi)$ can only have order $1$ or $\ell$.
    Pulling back through the representation shows $[M(E[\ell]): M^\psi]$ must divide $\ell$. 
    We also have
    \[[M^\psi : M(\mu_\ell)] = \frac{[M^\psi: M]}{[M(\mu_\ell): M]} = \frac{\sharp\psi(G_{M})}{(\ell-1)/2}.\]
    The curve $E/K$ is heavenly if and only if $[K(E[\ell]):K(\mu_\ell)] = \ell$, which will happen if and only if $[M^\psi : M(\mu_\ell)] = 1$.
    This happens exactly when $\sharp\psi(G_{M}) = (\ell-1)/2$, implying $\psi$ is not surjective. 
\end{proof}

\begin{remark}
    It is worth noting here that Proposition \ref{prop: heavenly not surj} only shows $\trace\left(\rho(G_K)\right) \subseteq \left(\frac{2}{\ell}\right)\cdot \FF_\ell^{\times 2} \cup \{0\}$ for totally balanced curves.
    This does not imply that $\sharp\trace\left(\rho(G_K)\right) = (\ell-1)/2$ exactly.
    However, under the assumptions of Theorem \ref{thm: CMheavenly iff not sur} $E$ has complex multiplication and we show that no other order is possible.
    If $\sharp\trace\left(\rho(G_K)\right) = \sharp\psi(G_{M}) = (\ell-1)/n < (\ell-1)/2$, then
    \[\frac{[M^\psi: M]}{[M(\mu_\ell): M]} = \frac{(\ell-1)/n}{(\ell-1)/2} < 1\]
    which is impossible.
    Hence in the case where $E$ is heavenly at $\ell$ and has complex multiplication, $\trace\left(\rho(G_K)\right) = \left(\frac{2}{\ell}\right)\cdot \FF_\ell^{\times 2} \cup \{0\}$ exactly.
\end{remark}

\section{Lifting Finiteness Results}\label{sec:finite-fibers}

    When classifying heavenly curves, it is a justifiable choice to consider only $\Qbar$-isomorphism classes. If $E/K_0$ is a curve, and $K/K_0$ and $K'/K_0$ are two finite extensions, the classes $[E\times_{K_0}K]
_K$ and $[E\times_{K_0}K']_{K'}$ are incomparable. The aim of this section is to prove that $\Qbar$ classes do not collapse together infinitely many classes of elliptic curves over cubic fields. In particular, we show that for any fixed pair $([E]_K,\ell) \in \mathcal{H}(K,1)$ that there does not exist an infinite family $\{([E_i]_{K_i},\ell)\}_{i\in \NN}$ such that $[K_i:\QQ]=[K:\QQ]$, $([E]_{K_i},\ell) \in \mathcal{H}(K_i,1)$, and $[E\times_K \Qbar]_{\Qbar} = [E_i\times_{K_i}\Qbar]_{\Qbar}$.

Recall that for any number field $K$, dimension $g > 0$, and rational prime $\ell$, $\mathcal{H}(K,g,\ell)$ denotes the set of $K$-isomorphism classes of abelian varieties of dimension $g$ which are heavenly at the prime $\ell$. Let $n = [K:\QQ]$, and for any $B \geq 1$ set 
$$\mathcal{H}(n,1)_B\defeq \bigcup_{\ell\geq B}\bigcup_{\substack{K\\ [K:\QQ] = n}} \{([A]_K,\ell) \in \mathcal{H}(K,1,\ell)\}.$$

We define an analogous set to keep track of $\Qbar$-isomorphism classes: 
$$\overline{\mathcal{H}}(n,1)_B \defeq \left\{[A\times_K \Qbar]_{\Qbar} : ([A]_K,\ell) \in \mathcal{H}(n,1)_B\right\}.$$

In set $\mathcal{H}^\circ(n,1)_B$ defined in the introduction adds the additional conditions that $K/\QQ$ is Galois, and no $A/K$ is realized as the base change of some $A_0/\QQ$ to $K$. We also consider the set

$$\overline{\mathcal{H}}^\circ(n,1)_B \defeq \left\{[A\times_K \Qbar]_{\Qbar} : ([A]_K,\ell) \in \mathcal{H}^\circ(n,1)_B\right\}.$$

We have an equivalence relation on $\mathcal{H}(n,1)_B$ defined as follows. For any $([A]_K,\ell), ([A']_{K'},\ell)$ that 
$$([A]_K,\ell) \sim ([A']_{K'},\ell) \quad \text{ if and only if } \quad A\times_K\Qbar \cong_{\Qbar} A'\times_{K'}\Qbar.$$

In the introduction, one of the main problems introduced is Question \ref{introq: H finite}: is $\sharp \mathcal{H}^\circ (3,1)_{11} < \infty$? However we then cite \cite[Thm. 3.1]{quadfields}, which shows that there are only finitely many $\Qbar$ isomorphism classes of heavenly elliptic curves at some fixed $\ell \geq 11$. Hence we now show when we are able to lift a finiteness result from the set $\overline{\mathcal{H}}^\circ(n,1)_B$ to the set of interest, $\mathcal{H}^\circ(n,1)_B$.

$\sharp\overline{\mathcal{H}}^\circ(n,1)_B < \infty \implies \sharp\overline{\mathcal{H}}^\circ(n,1)_B < \infty$, we will show the following map has finite fibers:

$$\alpha: \mathcal{H}^\circ(n,1)_B\to \overline{\mathcal{H}}^\circ(n,1)_B \qquad \qquad ([A]_K,\ell) \mapsto ([A]_{\Qbar}, \ell)$$

Let $\mathcal{U}$ and $\mathcal{B}$ denote the unit upper-triangular and Borel subgroup of upper-triangular matrices inside of $\GL_2(\FF_\ell)$. Given a heavenly elliptic curve $A/K$, proposition \ref{prop: heavenly upper triangular} shows there exists a basis of $E[\ell]$ as an $\FF_\ell$ vector space such that with respect to that basis, $\rho_A(G_K) \leq \mathcal{B}$ and $\rho_A(G_{K(\bmu_\ell)}) \leq \mathcal{U}$.

\begin{lemma}\label{lem: exectional j lifting}
    For $n$ odd and $B > 3$, there do not exist any $([E]_{\Qbar},\ell) \in \overline{\mathcal{H}}^\circ(n,1)_B$ with $j := j(E) \in \{0,1728\}.$
\end{lemma}

\begin{proof}
    Suppose there did exist some $([E]_{\Qbar},\ell) \in \overline{\mathcal{H}}^\circ(n,1)_B$ such that $j = 0$. Then there must exist some representative $E'/K$ in the $\Qbar$ isomorphism class of $E$ where $K$ is a Galois field of degree $n$, and $E'$ has CM field $\QQ(\bmu_6)$. For ease of notation, we will simply let $E$ be the correct representative. Then by \cite[Lem. 3.15]{BCS17}, it must be that $K(\bmu_6) \subseteq K(E[\ell])$.

    Let $\rho_E$ be the $G_K$-representation on the $\ell$-torsion of $E$, and let $\Gamma \defeq \Gal(K(E[\ell])/K)$. Since $E$ is heavenly at $\ell$ over $K$, we have that $\Gamma \cong \rho_E(G_K) \leq \mathcal{B}$ by Prop. \ref{prop: heavenly upper triangular}.

   Consider the subgroup $\Gamma_1 = \Gamma \cap\mathcal{U}$. It must be that $K(\bmu_\ell) = K(E[\ell])^{\Gamma_1}$, as $\rho_E(G_{K(\bmu_\ell)}) \leq \mathcal{U}$. Now,

   $$\sharp \Gamma_1 = [K(E[\ell]):K(\bmu_\ell)] = [K(E[\ell]):K(\bmu_6,\bmu_\ell)][K(\bmu_6,\bmu_\ell):K(\bmu_\ell)]$$

   Note that $[K(\bmu_6,\bmu_\ell):K(\bmu_\ell)] \leq [\QQ(\bmu_6): \QQ] = 2$. In addition, $\sharp\mathcal{U} = \ell$, so $\sharp\Gamma_1 = [K(E[\ell]):K(\bmu_\ell)]$ must be odd. This lets us conclude that $[K(\bmu_6,\bmu_\ell):K(\bmu_\ell)] = 1$ and hence $K(\bmu_6) \subseteq K(\bmu_\ell)$. We will show this conclusion leads to a contradiction. 

   As $\ell \geq B > 3$ then $\bmu_6 \not\subseteq \bmu_\ell$ and it must be that $[\QQ(\bmu_6,\bmu_\ell):\QQ(\bmu_\ell)] = 2$.

   Let $M = \QQ(\bmu_\ell) \cap K$, and let $r = [M:\QQ]$. Then,

   \begin{align*}
       [K(\bmu_\ell):\QQ] & = \frac{[K:\QQ][\QQ(\bmu_\ell):\QQ]}{[M:\QQ]}\\
       & = \frac{n(\ell-1)}{r}\\
       & = [K(\bmu_\ell):\QQ(\bmu_\ell)][\QQ(\bmu_\ell):\QQ]
   \end{align*}

   So $[K(\bmu_\ell):\QQ(\bmu_\ell)] = \frac{n}{r}$ which must be odd.

   However, $[K(\bmu_\ell) : \QQ(\bmu_\ell)] = [K(\bmu_\ell):\QQ(\bmu_6,\bmu_\ell)][\QQ(\bmu_6,\bmu_\ell):\QQ(\bmu_\ell)] = 2[K(\bmu_\ell):\QQ(\bmu_6,\bmu_\ell)]$, a contradiction.

We will now proceed with a similar proof for the case $j = 1728$. Suppose $([E]_{\Qbar}, \ell) \in \overline{\mathcal{H}}^\circ(n,1)_B$ with $j = 1728$. Again, for ease of notation we will let $E/K$ be the representative from the $\Qbar$ isomorphism class with CM field $\QQ(\bmu_4)$ and $K/\QQ$ a Galois extension of degree $n$. As $\ell \geq B > 3$, \cite[Lem. 3.15]{BCS17} shows $K(\bmu_4) \subseteq K(E[\ell])$.

As above let $\rho_E$ be the $G_K$-representation on the $\ell$-torsion of $E$, and let $\Gamma \defeq \Gal(K(E[\ell])/K)$. As $E/K$ is heavenly at $\ell$, Proposition \ref{prop: heavenly upper triangular} shows $\Gamma \cong \rho_E(G_K) \leq \mathcal{B}$. Consider $\Gamma_1 = \Gamma\cap \mathcal{U}$. As $\rho_E(G_{K(\bmu_\ell)}) \leq \mathcal{U}$, it must be that $K(\mu_\ell) = K(E[\ell])^{\Gamma_1}$. So,

\[\sharp\Gamma_1 = [K(E[\ell]):K(\bmu_\ell)] = [K(E[\ell]): K(\bmu_4, \bmu_\ell)][K(\bmu_4, \bmu_\ell):K(\bmu_\ell)]\]

As $[K(\bmu_4, \bmu_\ell):K(\bmu_\ell)] \leq [\QQ(\bmu_4):\QQ] =2$ and $\sharp\Gamma_1 = [K(E[\ell]):K(\bmu_\ell)]$ is odd, then $[K(\bmu_4, \bmu_\ell):K(\bmu_\ell)] =1$ and $K(\bmu_4) \subseteq K(\bmu_\ell)$. 

Note that as $\ell > 3$ $\bmu_4 \not\subseteq \bmu_\ell$ and hence $[\QQ(\bmu_4, \bmu_\ell): \QQ(\bmu_\ell)] =2$. Set $M \defeq \QQ(\mu_\ell)\cap K$ and let $r = [M:\QQ]$. Then,

\begin{align*}
       [K(\bmu_\ell):\QQ] & = \frac{[K:\QQ][\QQ(\bmu_\ell):\QQ]}{[M:\QQ]}\\
       & = \frac{n(\ell-1)}{r}\\
       & = [K(\bmu_\ell):\QQ(\bmu_\ell)][\QQ(\bmu_\ell):\QQ]
   \end{align*}

   So $[K(\bmu_\ell):\QQ(\bmu_\ell)] = \frac{n}{r}$ which must be odd.

   However, $[K(\bmu_\ell) : \QQ(\bmu_\ell)] = [K(\bmu_\ell):\QQ(\bmu_4,\bmu_\ell)][\QQ(\bmu_4,\bmu_\ell):\QQ(\bmu_\ell)] = 2[K(\bmu_\ell):\QQ(\bmu_4,\bmu_\ell)]$, a contradiction.
\end{proof}

For the next proposition, we must reduce to the case where the field of definition of the elliptic curves is of prime degree over $\QQ$. 

The following proof is based on \cite[Prop. 7]{MRequivalentConj}

\begin{proposition}\label{prop: finite fibers heavenly ex}
    Suppose $\overline{E}/\Qbar$ is an elliptic curve, $([\overline{E}]_{\Qbar},\ell) \in \overline{\mathcal{H}}^\circ(p,1)_B$ for $p$ prime. Suppose $B > p$, and $j:= j(E) \notin \{0,1728\}.$ Let $X$ be the fiber of $\alpha$ over $([\overline{E}]_{\Qbar},\ell)$.

    \begin{enumerate}
        \item[(a)] Suppose $K',K''$ are distinct fields of degree $p$ and $A'/K'$, $A''/K''$ are elliptic curves satisfying $([A']_{K'},\ell),([A'']_{K''},\ell) \in X$. Then $j(\overline{E}) \in \QQ$.
        \item[(b)] Let $A_0/\QQ$ be any elliptic curve with $j(\overline{E}) = j(A_0)$. for every $([A]_K,\ell)\in X$, it holds that $K \subseteq \QQ(A_0[\ell])$.
    \end{enumerate}
    In particular, the set $X$ must be finite.
\end{proposition}

\begin{proof}
    First we will argue that $(a)$ and $(b)$ are sufficient to show the finiteness of the fiber $X$. For the sake of contradiction, suppose $(a)$ and $(b)$ hold, but $X$ is infinite. Each fixed field $K$ may only contribute finitely many pairs of the form $([A]_K,\ell) \in X$, as the Shafarevich Conjecture ensures that there are finitely many $K$-isomorphism classes of elliptic curves with good reduction away from $\ell$. So if $X$ is infinite, it is because infinitely many distinct fields of degree $p$ are contributing elements of $X$. By $(a)$ it must be that $j \in \QQ$, hence there does exist some model $A_0/\QQ$ with $j(A_0) = j$. However part $(b)$ implies that $K \subset \QQ(A_0[\ell])$. It cannot be true that a finite extension of $\QQ$ contains infinitely many distinct subfields, and hence it must be that $\sharp X < \infty$.

    To prove part $(a)$, we will note that $A'$ and $A''$ both living in the fiber of $([\overline{E}]_{\Qbar},\ell)$ means $j(A') = j = j(A'')$. Hence $j \in K'\cap K'' = \QQ$. It is exactly in this step where the reduction to prime degree fields in necessary.

    We now prove $(b)$. We may only assume that there exists some $A_0/\QQ$ with $j(A_0) = j$ if $j \in \QQ$. By $(a)$, this must hold any time there are two distinct fields of order $p$ who produce an element of the fiber. Clearly by the Shafarevich Conjecture, if all of the pairs $([A]_K, \ell) \in X$ have the same field $K$, then it must be that $X$ is finite.

    So we may take $A_0/\QQ$ with $j(A_0) = j$. Let $([A]_K, \ell) \in X$ and suppose by way of contradiction that $K \not\subseteq \QQ(A_0[\ell])$. It must then be that $A_0\times_\QQ K$ and $A$ are quadratic twists of each other, as they are curves with the same $j$-invariant both defined over the same field $K$. So we have some quadratic character $\lambda: G_K \to \{\pm1\}$ such that $(A_0\times_\QQ K)^\lambda \cong_K A$. The goal of this proof is to extend this to a quadratic character on $G_\QQ$ so that we generate a heavenly elliptic curve $E_0/\QQ$ which must also be in the fiber $X$.

    We have not assumed $A_0$ is heavenly at $\ell$, but $A$ is. We can then use the fact that $A_0$ is a twist of $A$ to deduce 

    $$\rho_{A_0,\ell}\big\vert_{G_K} = \rho_{A,\ell}^\lambda = \begin{pmatrix}
        \lambda\chi^{i_1} & \star\\ 0 & \lambda\chi^{i_2}
    \end{pmatrix}$$

    Let $G_0 \defeq \Gal (\QQ(A_0[\ell])/\QQ)$. $K$ and $\QQ (A_0[\ell])$ are linearly disjoint by assumption, so $G_0 \cong \Gal (K(A_0[\ell])/K)$. So, $\rho_{A_0}(G_\QQ) \cong \rho_{A_0}(G_K)$. In addition, by the above observation, $\rho_{A_0}(G_K) = \rho_A^\lambda(G_K) \leq \mathcal{B}$ up to conjugacy. We can then conclude that $\rho_{A_0}(G_\QQ) \leq \mathcal{B}$. Hence there must exist maps $\phi_i:G_\QQ \to \FF_\ell^\times$ such that 

    $$\rho_{A_0} \sim \begin{pmatrix}
        \phi_1 & \star \\ 0 & \phi_2
    \end{pmatrix}$$

    These characters must also have the property that for $m = 1,2$ $\phi_m\big\vert_{G_K} = \lambda\chi^{i_m}$, so we will set $\varepsilon_m \defeq \phi_m\chi^{-i_m}$. Note that $\varepsilon_m\big\vert_{G_K} = \lambda$. We can then study the following commutative diagram:

% https://q.uiver.app/#q=WzAsNCxbMSwwLCJHXzAiXSxbMCwwLCJHX0siXSxbMSwxLCJcXG1hdGhiYntGfV9cXGVsbF5cXHRpbWVzIl0sWzIsMCwiR19cXG1hdGhiYntRfSJdLFszLDAsIiIsMCx7InN0eWxlIjp7ImhlYWQiOnsibmFtZSI6ImVwaSJ9fX1dLFsxLDAsIiIsMix7InN0eWxlIjp7ImhlYWQiOnsibmFtZSI6ImVwaSJ9fX1dLFsxLDIsIlxcbGFtYmRhIiwyXSxbMywyLCJcXHZhcmVwc2lsb25fbSJdLFswLDJdXQ==
\[\begin{tikzcd}
	{G_K} & {G_0} & {G_\mathbb{Q}} \\
	& {\mathbb{F}_\ell^\times}
	\arrow[two heads, from=1-1, to=1-2]
	\arrow["\lambda"', from=1-1, to=2-2]
	\arrow[from=1-2, to=2-2]
	\arrow[two heads, from=1-3, to=1-2]
	\arrow["{\varepsilon_m}", from=1-3, to=2-2]
\end{tikzcd}\]

We now want to show that $\varepsilon_1 = \varepsilon_2$. Let $\sigma \in G_\QQ$ be an arbitrary element, and denote the image under the natural projection as $\overline{\sigma} \in G_0$. Then let $\tilde{\sigma} \in G_K$ any element which maps to $\overline{\sigma}$. But then $\varepsilon_1(\sigma) = \lambda(\tilde{\sigma}) = \varepsilon_2(\sigma)$. As this is true for an arbitrary element of $G_\QQ$, we can conclude $\varepsilon_1 = \varepsilon_2$. So we can remove the indexing and refer to this map as $\varepsilon$. This must be a quadratic character as $\varepsilon(G_\QQ) \subseteq \lambda(G_K) = \{\pm 1\}$. So for $m = 1,2$, $\phi_m = \varepsilon\chi^{i_m}$, and we may construct $E_0 \defeq A_0^\varepsilon$, a quadratic twist. This gives us the following description of $\ell$-torsion representations.

$$\rho_{A_0} = \begin{pmatrix} \phi_1 & \star \\ 0 & \phi_2 \end{pmatrix} = \begin{pmatrix} \varepsilon\chi^{i_1} & \star \\ 0 & \varepsilon\chi^{i_2} \end{pmatrix}, \qquad \rho_{E_0} = \begin{pmatrix} \chi^{i_1} & \star \\ 0 & \chi^{i_2} \end{pmatrix}.$$

However, we have also constructed $\varepsilon$ such that $\varepsilon\big\vert_{G_K} = \lambda$, so $(A_0\times_\QQ K)^\lambda \cong_K (A_0^\epsilon \times_\QQ K) = E_0 \times_\QQ K$. Hence we can conclude that $E_0 \times_\QQ \Qbar \cong_{\Qbar} \overline{E}$ and $K(A[\ell^\infty]) = K\cdot\QQ(E_0[\ell^\infty])$.

We now wish to argue that $E_0$ is heavenly at the prime $\ell$. By Lemma \ref{lem:heavenly equiv def}, we must show that $[\QQ(E_0[\ell]):\QQ(\bmu_\ell)]$ is $\ell$-power degree, and $E_0$ has good reduction away from $\ell$.

When we restrict $\rho_{E_0}$ to $G_{\QQ(\bmu_\ell)}$ it is unit upper triangular, so $[\QQ(E_0[\ell]):\QQ(\bmu_\ell)]\big\vert\ell$. Now let $p \neq \ell$ be a distinct prime. We will argue that $E_0$ must have good reduction at $p$, using the criteria of N\'{e}ron-Ogg-Sharafevich \cite{Serre-Tate}, we will argue that $p$ is unramified in $\QQ(E_0[\ell^\infty])\QQ$. 

Arguing that $p$ is unramified for some arbitrary $\ell^n$ division field, will be sufficient to show $p$ is unramified in the composite $\QQ(E_0[\ell^\infty])$. So let $n > 1$ be an integer; set $F = \QQ(E_0[\ell^n])$ and $\mf{p}$ a prime of $F$ above $p$. Clearly $KF \subseteq K(A[\ell^\infty])$. As $A$ is heavenly at $\ell$, $KF/K$ must be unramified away from $\ell$. Hence any ramification of the prime $p$ is controlled by $K$. So $e_{\mf{p}\vert p} \leq [K:\QQ] = p$. In addition, $\QQ(\bmu_\ell)/\QQ$ is unramified at $p$ so $e_{\mf{p}\vert p}$ must divide $[F:\QQ(\bmu_\ell)]$, an $\ell$ extension.

However, we have assumed at the start of this proof that $\ell \geq B > p$. Hence the only way for $e_{\mf{p}\vert p}$ to be a power of $\ell$ \textit{and} less than or equal to $p$ is if $e_{\mf{p}\vert p} = 1$ as desired. This is sufficient to show $([E_0]_\QQ,\ell) \in \mathcal{H}(\QQ,1)$. This would then tell us that $([\overline{E}]_{\Qbar},\ell) \in \overline{\mathcal{H}}(\QQ,1)$, which contradicts the assumption that $([\overline{E}]_{\Qbar},\ell) \in \overline{\mathcal{H}}^\circ(p,1)_B$.
\end{proof}

\newpage
\section{Table of Non-Balanced Cases}
{\small
\begin{table}[!h]
    \centering
    \begin{tabular}{|c|c|c|}
    \hline
        $\ell$ & $(j_1,j_2)$ & $(i_1,i_2)$ \\ \hline
        13 & (0,1) & (0,1)\\
        {} & (0,2) & (0,1), (6,7)\\
        {} & (0,3) & (0,1), (4,9), (8,5)\\
        {} & (0,4) & (0,1), (9,4), (6,7), (3,10)\\
        {} & (0,6) & (0,1), (4,9), (8,5), (6,7), (10,3), (2,11)\\
        {} & (0,8) & (0, 1), (9, 4), (6, 7), (3, 10)\\
        {} & (0,9) & (0, 1), (4, 9), (8, 5)\\
        {} & (3,6) & (3, 10), (7, 6), (11, 2)\\
        {} & (0,12) & (0, 1), (4, 9), (8, 5), (9, 4), (1, 0), (5, 8), (6, 7), (10, 3), (2, 11), (3, 10), (7, 6), (11, 2)\\
        {} & (0,18) & (0, 1), (4, 9), (8, 5), (6, 7), (10, 3), (2, 11)\\
        {} & (6,12) & (9, 4), (1, 0), (5, 8), (3, 10), (7, 6), (11, 2)\\\hline
        19 & (0,3) & (6,13)\\
        {} & (0,6) & (6, 13), (15, 4)\\
        {} & (0,9) & (2, 17), (4, 15), (14, 5), (6, 13), (8, 11)\\
        {} & (0,12) & (6, 13), (15, 4)\\
        {} & (0,18) & (2, 17), (4, 15), (14, 5), (6, 13), (8, 11), (11, 8), (13, 6), (5, 14), (15, 4), (17, 2)\\\hline
        23 & (1,8) & (5,18)\\
        {} & (2,7) & (10,13)\\
        {} & (2,10) & (2,21), (13,10)\\
        {} & (2,16) & (16,7), (5,18)\\
        {} & (4,14) & (10,13), (21, 2)\\ \hline
        29 & (1,2) & (19,10)\\
        {} & (0,4) & (21,8), (7,22)\\
        {} & (2,4) & (5,24), (19,10)\\
        {} & (0,8) & (21, 8), (7, 22)\\
        {} & (1,8) & (25,4)\\
        {} & (2,7) & (22,7)\\
        {} & (3,6) & (19,10)\\
        {} & (0,12) & (21,8), (7,22)\\
        {} & (4,8) & (12,17), (5,24), (26,3), (19,10)\\
        {} & (2,16) & (25,4), (11,18)\\
        {} & (4,14) & (8,21), (22,7)\\
        {} & (6,12) & (5, 24), (19, 10)\\ \hline
        37 & (0,3) & (12,25) \\
        {} & (0,4) & (9, 28), (27, 10)\\
        {} & (0,6) & (12,25), (30,7)\\
        {} & (0,8) & (9, 28), (27, 10)\\
        {} & (0,9) & (20, 17), (32, 5), (8, 29), (12, 25), (16, 21), (28, 9), (4, 33)\\
        {} & (0,12) & (12, 25), (21, 16), (30, 7), (3, 34), (9, 28), (33, 4), (27, 10), (15, 22)\\
        {} & (0,18) & (20, 17), (32, 5), (8, 29), (2, 35), (14, 23), (26, 11), (12, 25), (16, 21), (30, 7),\\
        {} & {} & (34, 3), (28, 9), (4, 33), (10, 27), (22, 15)\\\hline
        43 & (2,6) & (16,27), (37,6)\\\hline
        47 & (2,6) & (6, 41), (29, 18)\\
        {} & (1,8) & (41, 6)\\
        {} & (2,16) & (18,29), (41,6)\\ \hline
        59 & (2,7) & (26, 33) \\
        {} & (4,14) & (26,33), (55,4) \\ \hline
        73 & (0,9) & (56, 17), (32, 41), (8, 65)\\
        {} & (0,18) & (56, 17), (32, 41), (8, 65), (20, 53), (68, 5), (44, 29)\\ \hline
        
    \end{tabular}
    \caption{Non-Balanced Cases from Prop. \ref{prop: B(3,1)}}
    \label{tab:survivingcase}
\end{table}}

\clearpage

\printbibliography

\end{CJK*}

\end{document}

%%% Local Variables:
%%% mode: LaTeX
%%% TeX-master: t
%%% End: